\documentclass[1p]{elsarticle}

\usepackage[ansinew]{inputenc}
\usepackage[T1]{fontenc}

\usepackage{multirow,bigdelim}   
\usepackage{graphicx}
\usepackage{amssymb}
\usepackage{textcomp}
\usepackage{amscd}
\usepackage{amsthm}
\usepackage{amsmath}
\usepackage[usenames]{color}
\definecolor{sonnengelb}{rgb}{0.9,0.5,0} 
\definecolor{brown}{rgb}{0.5,1,0}
\usepackage{url}

\usepackage{algorithm,algorithmic}

\newtheorem{Definition}{Definition}[section]
\newtheorem{Example}{Example}[section]
\newtheorem{thm}{Theorem}

\newtheorem{corollary}{Corollary}

\newtheorem{Problem}{Problem}

\newtheorem{Proposition}{Proposition}
\newproof{Proof}{Proof}
\newtheorem{Remark}{Remark}
\theoremstyle{definition}

\def\mdots{\vbox{\baselineskip=3pt \lineskiplimit=0pt 
\kern2pt \hbox{.}\hbox{.}\hbox{.}}}
\newcommand{\myvdots}{\phantom{1}\llap{\mdots\hspace{0.25ex}}}
\newcommand{\bigdots}{\multirow{2}{*}{\vdots}}
\newcommand{\ob}[1]{\begin{array}{|c|}\hline ~\makebox[0pt]{$#1$}~\end{array}}
\newcommand{\mi}[1]{\begin{array}{|c|}~\makebox[0pt]{$#1$}~\end{array}}
\newcommand{\un}[1]{\begin{array}{|c|}~\makebox[0pt]{$#1$}~\\\hline\end{array}}
\newcommand{\LiKl}[1]{\delimitershortfall=-4pt\ldelim({#1}{0pt}}
\newcommand{\ReKl}[1]{\delimitershortfall=-4pt\rdelim){#1}{0pt}}

\begin{document}
\begin{frontmatter}
\title{The Connection between the Number of Realizations for Degree Lists and Majorization}
\author{Annabell Berger\fnref{fn1}}
\address{Department of Computer Science, Martin-Luther-Universit\"at Halle-Wittenberg, Germany}
\ead{berger@informatik.uni-halle.de}
\fntext[fn1]{This work was supported by the DFG Focus Program Algorithm Engineering, grant MU 1482/4-3.}

\begin{abstract}
The \emph{graph realization problem} is to find for given nonnegative integers $a_1,\dots,a_n$ a simple graph (no loops or multiple edges) such that each vertex $v_i$ has degree $a_i.$ Given pairs of nonnegative integers $(a_1,b_1),\dots,(a_n,b_n),$ (i) the \emph{bipartite realization problem} ask whether there is a bipartite graph (no loops or multiple edges) such that vectors $(a_1,...,a_n)$ and $(b_1,...,b_n)$ correspond to the lists of degrees in the two partite sets, (ii) the \emph{digraph realization problem} is to find a digraph (no loops or multiple arcs) such that each vertex $v_i$ has indegree $a_i$ and outdegree $b_i.$\\ 
The classic literature provides characterizations for the existence of such realizations that are strongly related to the concept of majorization. Aigner and Triesch (1994) extended this approach to a more general result for graphs, leading to an efficient realization algorithm and a short and simple proof for the Erd\H{o}s-Gallai Theorem. We extend this approach to the bipartite realization problem and the digraph realization problem.\\
Our main result is the connection between majorization and the number of realizations for a degree list in all three problems. We show: if degree list $S'$ majorizes $S$ in a certain sense, then $S$ possesses more realizations than $S'.$ We prove that constant lists possess the largest number of realizations for fixed $n$ and a fixed number of arcs $m$ when $n$ divides $m.$ So-called \emph{minconvex lists} for graphs and bipartite graphs or \emph{opposed minconvex lists} for digraphs maximize the number of realizations when $n$ does not divide $m$. 
\end{abstract}


\end{frontmatter}

\section{Introduction}

\paragraph{Realization problems}

In our paper we work with three related problems in the topic of realizing degree lists.

\begin{Problem}[digraph realization problem/loop-digraph realization problem]\label{problem:digraph-realization}
Given is a finite list of pairs $((a_1,b_1),\dots,(a_n,b_n))$
with $a_i,b_i\in \mathbb{Z}_0^+.$ Does there exist 
\begin{enumerate}
\item[(1)] a digraph $G$ without multiple arcs and loops, or
\item[(2)] a digraph $G$ without multiple arcs and at most one loop per vertex
\end{enumerate}
with vertices $v_1,\dots,v_n$ such that $d_G^-(v_i)=a_i$ and $d_G^+(v_i)=b_i$ for all $i$, where $d_G^-(v)$ and $d_G^+(v)$ denote the indegree and outdegree of v in $G$, respectively?
\end{Problem}
 
If the answer is ``yes'', then we call the list in case (1) a \emph{digraphic list} and in case (2) a \emph{loop-digraphic list}. We also call such a digraph $G$ a \emph{digraph realization} or \emph{loop-digraph realization}, respectively. We forbid (0,0) from the list. In case (1) we demand $0 \leq a_i,b_i \leq n-1$ and in case (2) $0 \leq a_i,b_i \leq n.$ It is possible to reduce the loop-digraph realization problem to a \emph{bipartite realization problem}. Namely, one only has to consider the adjacency matrix of a loop-digraph realization as the adjacency matrix of an undirected bipartite graph. Vectors $(a_1,...,a_n)$ and $(b_1,...,b_n)$ then correspond to the lists of degrees in the two partite sets.\\
Furthermore, we consider  the \emph{graph realization problem}. Here, we have to decide for a given list $(a_1,...,a_n)$ whether there exists a graph $G$ with vertex set ${v_1,...,v_n}$ such that $d_G(v_i)=a_i$ for all $i$, where $d_G(v)$ denotes the degree of $v$ in $G$. Analogously to the directed case, we call such a list a \emph{graphic list} and the graph $G$ a \emph{graph realization}. We forbid $0$ in the list.\\
There is a very simple connection between the graph realization problem and the digraph realization problem. Given a list $(a_1,...,a_n)$ with even sum, one constructs the list $((a_1,a_1),...,(a_n,a_n))$ and solves the digraph realization problem. A result of Chen \cite{Chen80} shows that each digraphic list of this form also possesses a symmetric digraph realization $G$; replacing each opposed pair of arcs in $G$ with an edge having the same endpoints shows that $(a_1,...,a_n)$ is graphic.\\
In the classical literature we find two different approaches to solve these realization problems. A recursive algorithm constructing a realization was developed by Havel and Hakimi \cite{Havel55,Hakimi65} for the undirected case and by Kleitman and Wang \cite{KleitWang:73} for the directed case. The so-called \emph{characterization of degree lists} is the second approach; it requires the evaluation of at most $n$ inequalities. Our work is strongly connected with this approach. For that reason we briefly survey these results in this section. There exist several different notions for these results. We decided to use the classical connection to the majorization relation $\prec$ on real vectors, which was introduced by Hardy, Littlewood and Polya \cite{Hardy34}. In our work we use $a$ and $b$ to denote nonnegative integer $n$-tuples $(a_1,\dots,a_n)$ and $(b_1,\dots,b_n)$, respectively, and $(a,b)$ to denote the list $((a_1,b_1),\dots,(a_n,b_n))$ of pairs. Furthermore, we need a notion for the lexicographical sorting of a list. 
{\it Lexicographic ordering} is the relation
$\geq_{lex}\subset \mathbb{Z}^+\times \mathbb{Z}^+$ defined by 
$$(a,b)\geq_{lex} (a',b') \Leftrightarrow a>a' \textnormal{ or } (a=a'~ \textnormal{ and } b>b')$$
;it is a total ordering. Hence, it is possible to number all pairs of a list so that $(a_i , b_i)\geq_{lex} (a_j,b_j)$ if and only if $i\leq j.$ We denote such a labeling of a list by \emph{lexicographically nonincreasing list}.

\begin{Definition}[Majorization]\label{Def:Majorization}
We define then \emph{majorization} relation $\prec$ on $\mathbb{Z}^+\times \mathbb{Z}^+$ by $a \prec a'$ if and only if 
\begin{eqnarray*}
\sum_{i=1}^{k}a_i &\leq &\sum_{i=1}^{k}a'_i \text{ for all }k \in \{1,\dots,n-1\} \text{ and }\\
\sum_{i=1}^{n}a_i &= &\sum_{i=1}^{n}a'_i.\\
\end{eqnarray*}
Vector $a$ is said to be \emph{majorized} by $a'$ or $a'$ \emph{majorizes} $a$, respectively.  
\end{Definition}

Note that this classical definition involves the case that vector $a$ equals vector $a'$. Next we consider a special type of lists with respect to the majorization relation. \emph{Graphic} (\emph{loop-digraphic, digraphic}) \emph{threshold lists} are graphic lists (loop-digraphic lists, digraphic lists) that only have a unique graph realization (loop-digraph realization, digraph realization) the so-called \emph{threshold graph} (\emph{threshold loop-digraph},\emph{threshold digraph}). For the undirected case a definition was given by Chvatal and Hammer \cite[Theorem 1]{Chvatal1977}, for threshold loop-digraphs by Hammer et al. (there called \emph{difference graphs}) \cite[Theorem 2.3.]{Hammer1990}, and for the directed case this definition was given in our Phd-thesis \cite{Berger2011}.\\
It is important to mention that one can find several equivalent characterizations of threshold (loop)-(di)graphs. The first two authors has not stated the uniqueness of the (loop)-(di)graph realizations explicitly; instead Chvatal and Hammer stated in \cite[Theorem 1(ii)]{Chvatal1977} the non-existence of a so-called \emph{swap} or \emph{switch} in a threshold graph, i.e.\ one cannot find edges $\{a,b\}$, $\{c,d\}$ and non-edges $\{a,c\}$ and $\{b,d\}$ together. This is indeed equivalent to the uniqueness of a suitable graph realization. To see that one has to apply a classic result of Peterson \cite{Petersen1891} from 1891; for a given degree list each graph realization can be yield from another graph realization by a series of swaps. Analogous results can be found for loop-digraphs and digraphs with the only difference that one needs a further operation to yield each digraph realization, i.e.\ the reorientation of a directed $3$-cycle see Rao et al.~\cite{RaoJanaBandyopadhyay96}.\\
However, (loop)-(di)graphic threshold lists play an important role for all three realization problems, because they cannot be majorized by another (loop)-(di)graphic list. More exactly, there does not exist a graphic list $a$ that majorizes a graphic threshold list $a'$ \cite[Theorem 3.2.2.8]{MahadevPeled95}. The same is true for each loop-digraphic threshold list $(a',b)$; there is no loop-digraphic list $(a,b)$ such that $a$ majorizes $a'$. There also exist the analogous variant for digraphic threshold lists. We believe that the last two statements for loop-digraphs and digraphs are not stated explicitly in the literature. A reader can easily understand these properties, if (s)he realizes the strong connection between threshold lists and Ferrers matrices, which we point out after the next definition.

\begin{Definition}[loop-digraphic Ferrers matrix]\label{Definition:loop-digraphicFerrersMatrix}
Given a paired list $(a,b)$ such that $a$ is nonincreasing. We construct an $n \times n$-matrix $F$ defined by 
\begin{eqnarray*}
F_{ij}&:=&\begin{cases}
1 &\text{if}~ j\leq b_i\\
0 &\text{if}~ j> b_i.\\ 
\end{cases}\\
\end{eqnarray*}
$F$ is called the \emph{loop-digraphic Ferrers matrix} for list $(a,b)$. 
\end{Definition}

Ferrers introduced this notion in the context of partitions (row and column sums can be interpreted as a partition of $\sum_{i=1}^n a_i$ and its dual partition). For more information see Sylvester \cite{Syl:1882} or the overview about Norman Macleod Ferrers in \cite{NMF}. Clearly, a loop-digraphic Ferrers matrix can be seen as the adjacency matrix of a loop-digraph with loop-digraphic list $(a,b).$ There cannot exist a swap in this loop-digraph, because the matrix consists top down of nonincreasing blocks. Hence, a loop-digraphic Ferrers matrix is a unique realization of list $(a,b)$; thus $(a,b)$ is a loop-digraphic threshold list.\\
Let us consider the $n \times n$-adjacency matrix $A$ of a threshold loop-digraph $G$. Is it possible that such a matrix is not a loop-digraphic Ferrers matrix? It depends from the labeling of the vertices in $G$. If we demand that the row and colum sums are lexicographically nonincreasing, then $A$ is a loop-digraph Ferrers matrix, because the existence of entries $A_{ik}=0$, $A_{il}=1$ for $k<l$ would imply the existence of a swap in contradiction to the uniqueness of $A$. If we do not demand this kind of sorting for the vertices of a threshold loop-digraph then the adjacency matrix is not at all a Ferrers matrix, consider for example a matrix of ones where one column consists of zeros, i.e.\ the adjacency matrix of a threshold loop-digraph. We already stated all these observations in our PhD-thesis~\cite{Berger2011} without a proof, because we thought this is folklore. Cloteaux at al. \cite[Theorem 1]{Cloteaux12} have written a paper to prove these things from our PhD-thesis.\\
However, for us here it is important to understand that the loop-digraphic list of a Ferrers matrix is always a loop-digraphic threshold list. 
More exactly, let $a'_{i}$ denote the $i$th column sum of $F$, then $(a',b)$ is a loop-digraphic threshold list. We call $(a',b)$ the \emph{corresponding loop-digraphic threshold list} of $(a,b)$. Clearly, each list (taking our definition with $0\leq b_i \leq n$) possesses a loop-digraphic Ferrers matrix. In the context of partitions of integers, $a'$ is the conjugate partition of $b.$ We cite the classical characterization result from Gale and Ryser \cite{Ga:57,Ry:57} on loop-digraphs or bipartite graphs, respectively. The idea is to determine for a list $(a,b)$ its loop-digraphic Ferrers matrix with degree list $(a',b)$. If $a \prec a'$, then $(a,b)$ is a loop-digraphic list.

\begin{thm}[Gale, Ryser] \label{Theorem:RyserGale}
Let $(a,b)$ be a paired list such that $\sum a_i=\sum b_i$, the list $a$ is nonincreasing, and $(a',b)$ the corresponding loop-digraphic threshold list. Then $(a,b)$ is loop-digraphic if and only if $a\prec a'$.
\end{thm}

Although a list $(a,b)$ always has a loop-digraphic threshold list, it does not need to be a loop-digraphic list. Consider for example list $((3,3),(1,3),(2,0))$ and its corresponding threshold list $((2,3),(2,3),(2,0)).$ Since $(2,2,2) \prec (3,1,2)$, $((3,3),(1,3),(2,0))$ is not a loop-digraphic list. Next we give the characterization result for digraphs. Again, we introduce Ferrers matrices in this context.

\begin{Definition}[digraphic Ferrers matrix]\label{DefinitionFerrersMatrixDigraphs}
Let $(a,b)$ be a paired list such that $a$ is nonincreasing. We construct an $n \times n$ matrix $F$ by 
\begin{eqnarray*}
F_{ij}&:=&\begin{cases} 
1 &\text{if}~ (j\leq b_i \textnormal{ and } j<i) \textnormal{ or } (j\leq b_i+1 \textnormal{ and } j>i)\\
0 &\text{otherwise.}\\ 
\end{cases}\\
\end{eqnarray*}
$F$ is called the \emph{digraphic Ferrers matrix} for list $(a,b)$.
\end{Definition}

Let $a'_i$ denote the $i$th column sum of $F$. With an analogous argumentation as after Definition~\ref{Definition:loop-digraphicFerrersMatrix} it can easily be seen that the digraphic list $(a',b)$ of the digraphic Ferrers matrix is a digraphic threshold list. We call list $(a',b)$ the \emph{corresponding digraphic threshold list} for list $(a,b).$ In contrast to the analogous loop-digraph realization problem, a list $(a,b)$ can possess several different digraphic threshold lists. The reason is that different sortings of $(a,b)$ with respect to nonincreasing $a_i$ can lead to different digraphic Ferrers matrices. Nevertheless, the characterization theorem is true for every digraphic threshold list of $(a,b)$. This result is relatively unknown until today. In 1966 Chen \cite{Chen:66} only proved the characterization theorem for lexicographically nonincreasing lists (and its Ferrers matrices). We later proved this theorem in its general form in \cite{Berger2014}. Subsequently we found out that Anstee already observed this in 1982 in the context of bipartite $(0,1)$-matrices with several restrictions \cite{Anstee82}. 

\begin{thm}[Fulkerson, Chen, Anstee \cite{Fulk:60,Chen:66,Anstee82}] \label{Theorem:FulkersonChenAnstee}
Let $(a,b)$ be a paired list such that $\sum a_i=\sum b_i$, the list $a$ is nonincreasing, and $(a',b)$ the corresponding digraphic threshold list. Then $(a,b)$ is digraphic if and only if $a\prec a'$.
\end{thm}

Since digraphic lists of the form $(a,a)$ have realizations by symmetric digraphs \cite{Chen:66}, applying Theorem~\ref{Theorem:FulkersonChenAnstee} to $(a,a)$ yields the well-known result of Erd\H{o}s and Gallai for graphic lists. First, we define the \emph{corresponding graphic threshold list} $a'$ of list $a$ to be the vector $a'$ obtained from the digraphic threshold list $(a',a)$ of $(a,a)$ in Definition~\ref{DefinitionFerrersMatrixDigraphs}. 

\begin{thm}[Erd\H{o}s, Gallai \cite{ErGa:60}] \label{Theorem:ErdGallai}
Let $a$ be a nonincreasing integer list with even sum, and $a'$ its corresponding threshold graphic list. Then $a$ is graphic if and only if $a\prec a'$.
\end{thm}

A short constructive proof was given by Tripathi et al. \cite{Tripathi2010}. An integer list can possess a graphic threshold list but is not graphic. Consider for example integer list $(3,3,1,1).$ To the best of our knowledge the complexity status for the problems of counting the realizations for a given (integer) list, is open for all three realization problems. In contrast, for a related problem -- \emph{the bipartite multigraph realization problem}, i.e.\ counting the number of bipartite multigraphs for a given list, is known to be $\sharp P$-hard \cite{Dyer95}. On the other hand, there has been a lot of work on counting graph realizations and loop-digraph realizations for a given (integer) list in the context of approximation, approximation algorithms, and randomized algorithms; see for example  \cite{Bender78,brendan91,JerrumSinclairVigoda04,Bezakova06,RSA:10,Barvinok08,Canfield08}. However, in our work we concentrate on a connection between different (integer) lists with respect to the number of their realizations.

\paragraph{Our Contribution}
There are different approaches to prove Theorems \ref{Theorem:RyserGale}, \ref{Theorem:FulkersonChenAnstee}, and \ref{Theorem:ErdGallai} using for example network flow theory or induction. In the next section we give generalizations for these theorems. The first one for graphs is from Aigner and Triesch \cite{Aigner1994}. We follow their approach and extend it to the two directed cases. The proofs are very simple and can easily be used for proving these theorems. Additionally, they give a further type of efficient algorithms for constructing suitable realizations. We believe these are the shortest known proofs for the characterization of degree lists.\\
It is well-known that threshold lists possess exactly a unique realization. We found a connection that is a generalization of this insight. In particular we found that for integer lists $a$ and $a'$ such that $a'$ majorizes $a$, the number of graph realizations of $a$ is larger than the number of realizations of $a'.$ We prove an analogous result for all realization problems, using paired lists; for $(a,b)$ and $(a',b)$ such that $a'$ majorizes $a$, the number of (loop)-digraph realizations of $(a,b)$ is larger than the number of realizations of $(a',b).$ To get these results there are several restrictions. For the loop-digraph problem and the graph realization problem vector $a$ must be nonincreasing, and for the digraph problem $(a,b)$ must be lexicographically nonincreasing.\\
It is sometimes possible to find a lower bound on the number of realizations for lists with exponentially many realizations. We do not give a formal result, but we refer to Example~\ref{Example:ExponentialManyLoopdigraphRealizations} to show how such a computation works. Such a result can be very important for the well-known \emph{uniform sampling problem}. In network analysis it is often desirable to sample uniformly at random from among the realizations of a given list. There exist two well-known Metropolis Markov chains for an approximate uniform sampler. The only known efficient chain for the general case is a perfect matching sampler given by Jerrum et al. \cite{JerrumSinclairVigoda04} in reducing a realization problem to a perfect matching problem via Tutte \cite{Tutte52}. This chain has large polynomial running times \cite{BezakovaBhatnagarVigoda07} and is far from being applicable in practice.\\
A more useful approach is the so-called \emph{swap algorithm} \cite{Petersen1891} that was proven to be efficient for several graph classes like (half)-regular (di)graphs \cite{KannanTetaliVempala99,CooperDyerGreenhill07,CooperDyerGreenhill12,Greenhill11,MiklosES13}, but in general its efficiency is unknown. (There exist other approaches than the use of Markov chains; for an overview consider the introduction in \cite{Greenhill11}.) However, to sample realizations there exists no efficient solution for practitioners. On the other hand, the running times of the sampling problem depend on the number of realizations of the input list. If there are only a polynomial number of realizations, one could enumerate all of them and solve the problem by choosing one uniformly at random, avoiding the mentioned approaches.\\
Our results lead to several insights with respect to this problem. In some cases, one can easily determine that there must be exponentially many realizations. Interestingly, we observe that lists that are `near'-threshold lists have not so many realizations as `near'-regular lists, for which the problem has been solved. So, it seems that our results could lead to new ways for sampling realizations of degree lists for lists that are far from regular.\\
More formally, we prove for all three problems that so-called minconvex lists, which are generalizations of regular lists, possess the largest number of realizations under all realizations with fixed $n$ and $m.$ For digraphic lists the result is more complicated. Here, the largest number of realizations is achieved by a special type of minconvex list, i.e.\ opposed lists $(a,b)$ where vector $a$ is nonincreasing and vector $b$ is nondecreasing. Minconvex lists are in certain sense the `contrary threshold lists'. They are minimal in the majorization order and possess the largest number of realizations.

\paragraph{Overview} In Section \ref{Chapter:GeneralizationCharakterization}, we generalize the characterization Theorems \ref{Theorem:RyserGale} and \ref{Theorem:FulkersonChenAnstee}. The proofs lead to a new type of realization algorithm. In Section \ref{Chapter:Majorization}, we explore the connection between majorization and the number of realizations for a given degree list. Furthermore, we show that minconvex lists possess the largest number of realizations.

\section{Generalizations of Characterizations of Degree Lists}\label{Chapter:GeneralizationCharakterization}

Similar but not identical to Mahadev and Peled \cite[Definition 3.1.2]{MahadevPeled95} and Marshall and Olkin \cite{Marshall79} we define \emph{transfers} on integer lists. Let the $i$th {\it unit list} be the $n$-tuple $e_i$ having $1$ in coordinate $i$ and $0$ elsewhere.

\begin{Definition}[transfer]
For an integer list $a'$ with $a'_i\geq a'_j+2$ for $i,j$ such that $1 \leq i<j \leq n$, the list obtained from $a'$ by an {\it (i,j)-transfer} (written $t_{i,j}(a')$) is the list $a'-e_i+e_j$. Sometimes, we simply use the term \emph{transfer} without specifying the indices.
\end{Definition} 

We repeat a classical result of Muirhead \cite{Muirhead03} in a version by Mahadev and Peled \cite[Theorem 3.1.3]{MahadevPeled95} with a small distinction for our specific investigations. This proof gives also an algorithm for obtaining a nonincreasing list from a list that majorizes it. 

\begin{thm}[Muirhead 1902~\cite{Muirhead03}]\label{MuirheadLemma}
If $a$ and $a'$ are nonnegative integer lists such that $a$ is nonincreasing and $a\prec a'$, then $a$ can be obtained from $a'$ by $\kappa$ successive unit transfers, where $\kappa=\frac{1}{2}\sum_{j=1}^n |a_j'-a_j|$. 
\end{thm}

\begin{proof}
For a list $a$, we denote by list $p_j(a)$ the $j$th partial sum, $\sum_{i=1}^j a_i$. We use induction on $\kappa=\frac{1}{2}\sum_{j=1}^n |a_j'-a_j|$; by the definition, each summand is nonnegative.  When the sum is $0$, the lists are the same.  When positive, let $\ell$ be the first index such that $p_{\ell}(a)<p_{\ell}(a')$; by definition, $\ell<n$. Since this is the first position with a difference, $a'_{\ell}>a_{\ell}$. Since $p_n(a)=p_n(a')$, there is a least index $k$ such that $a'_k<a_k$; note that $k>\ell$. Since $a$ is nonincreasing, we get $a'_{\ell}>a_{\ell}\geq a_{k}>a'_{k}$ and thus
\begin{gather} a'_k \geq a'_{\ell+1}+2. \tag{*} \end{gather}
Let $a''$ be a list obtained from $a'$ by an $(\ell,k)$-transfer, i.e.\ $a''=t_{\ell,k}(a')$.  The partial sums of $a''$ are the same as for $a'$, except that $p_j(a'')=p_j(a')-1$ for $\ell\le j<k$.  However, $a'_j\ge a_j$ for $\ell<j<k$ yields $p_j(a')>p_j(a)$ for $\ell\le j<k$. Hence $a\prec a''\prec a'$. The values $|a''_j-a_j|$ are the same as $|a'_j-a_j|$, except for $\ell$ and $k$. Since $a'_{\ell}-a_{\ell}>0$, $|a''_{\ell}-a_{\ell}|=|a'_{\ell}-1-a_{\ell}|=|a'_{\ell}-a_{\ell}|-1$. Since $a'_{k}-a_{k}<0$, $|a''_{k}-a_{k}|=|a'_{k}+1-a_{k}|=|a'_{k}-a_{k}|-1$. This yields $\frac{1}{2}\sum_{j=1}^n |a_j''-a_j|=\kappa -1$. By the induction hypothesis, subsequent transfers turn $a''$ into $a$ as desired.
\end{proof}

In contrast to the original proof we have only ordered the integers in $a$ whereas we omitted the ordering of the integers of $a'.$  This is in fact very important with respect to the digraph realization problem. It turns out that Muirhead's Lemma can be used for the construction of a graph realization, loop-digraph realization, and a digraph realization if one starts with the corresponding threshold list of a given graphic list, loop-digraphic list, or digraphic list, respectively. A crucial observation in constructing realizations is that applying any transfer to a graphic list yields a graphic list, since $d_G(v)\ge d_G(w)+2$ implies the existence of $u\in N(v)\setminus N(w)$ in a realization $G$, and $\{v,u\}$ can be replaced with $\{u,w\}$. The variant for graphic lists was first given by Aigner and Triesch \cite{Aigner1994} and later rediscovered by Mahadev and Peled \cite[Corollary 3.1.4]{MahadevPeled95}. Aigner and Triesch called the order `dominance' rather than `majorization', and several papers continue that usage.

\begin{thm}[Aigner, Triesch \cite{Aigner1994}]\label{Theorem:AignerTriesch}
If $a$ is a nonincreasing nonnegative integer list, and $a'$ is a graphic list with $a \prec a'$, then $a$ is graphic and can be obtained from $a'$ via unit transfers, passing only through graphic lists.
\end{thm}

We can consider this theorem as a generalization of Theorem \ref{Theorem:ErdGallai} by Erd\H{o}s and Gallai, since it is very easy to determine a graphic threshold list $a'$ majorizing an integer list $a$ by using the Ferrers matrix of Definition \ref{DefinitionFerrersMatrixDigraphs}. We also obtain an algorithm to construct a realization of a given graphic list. In the following, we want to extend this approach to loop-digraphic lists and digraphic lists. 

\begin{thm}\label{Theorem:DigraphSequences}
Let $(a',b)$ be a digraphic list, $(a,b)$ be a list with nonincreasing $a$, and $a \prec a'$. Then $(a,b)$ is a digraphic list and can be obtained from $(a',b)$ by unit transfers between $a'$ and $a$, passing only through digraphic lists.
\end{thm}

\begin{proof}
Integer list $a$ can be obtained from $a'$ by successive unit transfers via $\kappa$ integer lists $a^i$ (see Theorem~\ref{MuirheadLemma} by Muirhead). We show by induction on $\kappa$ that each $(a^i,b)$ is a digraphic list. Let $G^i$ be a digraph realization of list $(a^i,b)$ and let $(a^{i+1},b)$ obtained by a $(k,\ell)$-transfer, i.e. $a^{i+1}=t_{k,\ell}(a^i).$  With (*) in the proof of Theorem \ref{MuirheadLemma}, we have $a^{i}_k \geq a^{i}_{\ell}+2.$ Hence, there exist two different vertices $j,j' \in V(G^i)$ with $j\neq k$, $j' \neq k$ where $(j,k),(j',k) \in A(G^i)$ and $(j,\ell),(j',\ell) \notin A(G^i).$ Clearly, for at least one of the vertices $j$ and $j'$ we find $j \neq \ell$ or $j'\neq l.$ W.l.o.g.\ we assume $j \neq \ell.$ Then digraph $G^{i+1}$ with arc set $A(G^{i+1}):=(A(G^i) \setminus \{(j,k)\}) \cup \{(j,\ell)\}$ is a digraph realization of list $(a^{i+1},b).$
\end{proof}

\begin{Remark}\label{Remark:TheoremDigraphSequences}
In Theorem \ref{Theorem:DigraphSequences} it is not necessary to sort vector $a$ in the digraphic list $(a,b)$, if it $a$ can be obtained by a unit transfer from $a'$. The reason is that a nonincreasing $a$ is only necessary for the use of Muirhead's Theorem \ref{MuirheadLemma}. The construction of a digraph does not need this sorting.
\end{Remark}

Clearly, this theorem generalizes Theorem \ref{Theorem:FulkersonChenAnstee} of Fulkerson, Chen and Anstee if we demand that $(a',b)$ is the digraphic threshold list of $(a,b).$ Furthermore, we get a new digraph realization algorithm for $(a,b),$ which first determines for a given list $(a,b)$ its threshold list $(a',b)$. In the case that $(a,b)$ is a digraphic list ($a\prec a'$), the algorithm determines with Theorem \ref{MuirheadLemma} and the proof of Theorem \ref{Theorem:DigraphSequences} step by step a list of digraphic lists $(a^1,b),(a^2,b),\dots,(a^r,b)$ where $(a^1 ,b)=(a',b)$, $(a^r,b)=(a,b)$, and corresponding digraph realizations $G^1,G^2,\dots,G^r.$ We can use a slight extension of the proof from Theorem \ref{Theorem:DigraphSequences} for the case of loop-digraphic lists. (We only have to delete the demands $j\neq k, j' \neq k$ and $j \neq \ell.$) We get a generalization of Theorem \ref{Theorem:RyserGale} by Gale and Ryser.

\begin{thm}\label{Theorem:LoopDigraphSequences}
Let $(a',b)$ be a loop-digraphic list, $(a,b)$ be a list with nonincreasing $a$, and $a \prec a'$. Then $(a,b)$ is a loop-digraphic list and can be obtained from $(a',b)$ by unit transfers between $a'$ and $a$, passing only through loop-digraphic lists.
\end{thm}

\begin{Remark}\label{Remark:TheoremLoopDigraphSequences}
In Theorem \ref{Theorem:LoopDigraphSequences} it is not necessary to sort loop-digraphic list $(a,b)$, if $a$ can be obtained by a unit transfer from vector $a'$. The reason is that a nonincreasing vector $a$ is only necessary for the use of Muirhead's Theorem \ref{MuirheadLemma}. The construction of a loop-digraph does not need this sorting.
\end{Remark}

We introduce a special notion for sequences $(a^1,b),(a^2,b),\dots,(a^r,b)$ of (loop)-digraphic lists as they appear in Theorems~\ref{Theorem:DigraphSequences} and \ref{Theorem:LoopDigraphSequences}.

\begin{Definition}[transfer path]
Sequences $(a^1,b),(a^2,b),\dots,(a^r,b)$ of (loop)-digraphic lists $(a^i,b)$, where $a_i$ yields $a^{i+1}$ by a unit transfer, are called \emph{transfer paths}. We denote the value $(r-1)$ as the \emph{length} of a transfer path.
\end{Definition}

Note that the last three theorems give a proof for the existence of at least one transfer path between two (loop)-digraphic lists, where one majorizes the other one. Clearly, there often exist several different such paths with different lengths. We give an example.

\begin{Example}\label{Example:transferPaths}
We consider the two loop-digraphic lists $(a,b)$ and $(a',b)$ with $a=(2,2,2,0)$, $b=(1,1,2,2)$ and $a'=(4,2,0,0)$. Then we find the following transfer paths.
\begin{description}
\item[(1)] ${(a',b)}{,}{~(a^2,b)}{=}{(3,1)}{,}{(3,1)}{,}{(0,2)}{,}{(0,2)}{,}{~(a^3,b)}{=}{(3,1)}{,}{(2 ,1)}{,}{(1,2)}{,}{(0,2)}{,}{~(a,b)}$
\item[(2)] ${(a',b)}{,}{~(a^2,b)}{=}{(3,1)}{,}{(2,1)}{,}{(1,2)}{,}{(0,2)}{,}{~(a,b)}$
\item[(3)] ${(a',b)}{,}{~(a^2,b)}{=}{(4,1)}{,}{(1,1)}{,}{(1,2)}{,}{(0,2)}{,}{~(a^3,b)}{=}{(3,1)}{,}{(2,1)}{,}{(1,2)}{,}{(0,2)}{,}{~(a,b)}$
\item[(4)] ${(a',b)}{,}{~(a^2,b)}{=}{(3,1)}{,}{(3,1)}{,}{(0,2)}{,}{(0,2)}{,}{~(a^3,b)}{=}{(2,1)}{,}{(3,1)}{,}{(1,2)}{,}{(0,2)}{,}{~(a,b)}$
\end{description}
Transfer path (2) is the one constructed in the proof of Theorem \ref{Theorem:LoopDigraphSequences}.
\end{Example}

Note that it is possible that $(a^i,b)$ is not lexicographically nonincreasing (see $(a^3,b)$ in our fourth transfer path). This is indeed important, because several digraphic lists do only possess digraphic threshold lists which are not nonincreasing.

\section{The Connection between Majorization and the Number of Realizations}\label{Chapter:Majorization}

The main result of this section is to see that the number of (loop)-digraph realizations for a given (loop)-digraphic list is smaller than the number of (loop)-digraph realizations for each majorized list. Indeed we prove that the number of (loop)-digraph realizations increases for each (loop)-digraphic list in a transfer path. Since, each pair of lists possesses at least one transfer path as we showed in the last section our claim can be proven by this approach. Briefly we want to give an intuition for this idea. A threshold (loop)-digraphic list can only majorize other (loop)-digraphic lists and has only a unique realization. We asked if this property can be extended: Do lists `near by' threshold lists possess only few realizations in contrast to lists with a `large distance' to a threshold list. This conjecture is true in a certain sense and can give for each list a `feeling' of the number of realizations. Following this result, we can observe that \emph{regular lists} ($a_i=b_i=d$ for all $i \in \{1,\dots,n\}$) possess the largest number of realizations in the set of all lists with $n$ pairs and fixed $m:=\sum_{i=1}^n a_i.$ This is true, since regular lists do not majorize any other list. Note that regular lists only exist in the case when $n$ is a factor of $m.$ It turns out, that (loop)-digraphic lists with a minimum sum of squared indegrees and a minimum sum of squared outdegrees  possess the maximum number of realizations in the set of all lists with fixed $n$ and $m.$ We call such (loop)-digraphic lists \emph{minsquare lists} or more general \emph{minconvex lists}. To see the connection between sum of squares (convex functions) and majorization consider the famous result of Polya, Littlewood and Hardy \cite{Hardy29}. 

\begin{thm}[Hardy, Polya, Littlewood \cite{Hardy29}]\label{Theorem:Hardy}
Let $g:\mathbb{R}\mapsto \mathbb{R}$ be an arbitrary convex function and $a,a' \in \mathbb{R}^n.$
We have $$\sum_{i=1}^{n}g(a_i)\leq \sum_{i=1}^{n}g(a_i')$$
if and only if $a \prec a'.$
\end{thm}

There is a nice classical connection which is worth to be mentioned here. Let us define the following set $DS$ with
$$DS:=\{\textnormal{a is a graphic list.}\}$$
Then the \emph{polytope of graphic lists} is defined as the convex hull $D_n$ of $DS.$ Koren \cite{Koren1973213} proved 
the following theorem.

\begin{thm}[Koren \cite{Koren1973213} ]\label{Theorem:convexeHuelle}
A graphic list is an extreme point of $D_n$ if and only if it is a threshold list.
\end{thm}

A generalization of Koren's condition for multigraphs with bounded multiplicity of edges, bipartite graphs and directed graphs was given by Isaak and West \cite{IsaakW10}. The following ideas are folklore but we want to point out the interesting connection between majorization and threshold lists. Hence, each graphic list $a$ can be constructed by a convex combination of threshold lists $x_i:=(x_{i1},\dots,x_{in}),$ i.e.\ $a=\sum_{i=1}^k \lambda_i x_{i}$ where $\sum_{i=1}^k \lambda_i=1$ $(\lambda_i > 0).$ For an arbitrary convex function $g\colon\mathbb{R} \mapsto \mathbb{R}$---using the inequality of Jensen---we get
$$\sum_{j=1}^{n}g(a_j)=\sum_{j=1}^{n}g(\sum_{i=1}^k\lambda_ix_{ij})
\leq\sum_{j=1}^{n}\sum_{i=1}^k \lambda_i g(x_{ij})
 =\sum_{i=1}^{k}\lambda_i \sum_{j=1}^n  g(x_{ij}).$$
We can conclude that there exists at least one $x_i$ with $\sum_{j=1}^{n}g(x_{ij})\geq \sum_{j=1}^{n}g(a_j).$ Assume this is not the case. Then we have 
$$\sum_{i=1}^{k}\lambda_i \sum_{j=1}^n g(x_{ij})<\sum_{i=1}^{k}\lambda_i \sum_{j=1}^n g(a_j)=\sum_{j=1}^{n}g(a_j)$$
in contradiction to our above inequality. By Theorem \ref{Theorem:Hardy} there exists an $x_i$ with $a \prec x_i.$ Hence, the connection of majorization and threshold lists follows from the property that threshold lists are extreme points of the polytope of graphic lists. It is very natural to ask for the number of realizations of graphic lists that are not threshold lists. Clearly, the number is larger than one. But a very intuitive idea for us was to ask for a more well-founded connection. We start our discussion with loop-digraphic lists, because the proofs are simpler than in the case of digraphic lists and graphic lists. Nevertheless, we try to develop our proofs in an analogous way for all cases.

\subsection{The number of loop-digraph realizations and majorization}

In a first step, we define the set $R_1(a,b)$ of all loop-digraph realizations for a given loop-digraphic list $(a,b)$ and denote by $N_1(a,b):=|R_1(a,b)|$ the number of loop-digraph realizations of $(a,b).$ We want to give a connection between a unit $(i,j)$-transfer on a loop-digraphic list $(a',b)$ yielding $(a,b)$ and a corresponding operation on a loop-digraph realization $G'$ of $(a',b)$ leading to a loop-digraph realization $G$ of $(a,b).$ Let us now consider the adjacency matrix $A'$ for an arbitrary loop-digraph realization of list $(a',b).$ Then $a'_i,a'_j$ are the sums of the $i$th or $j$th columns, respectively. Since $a'$ yields $a$ by a unit transfer we have $a'_i \geq a'_j+2.$ Hence, it is possible to shift $|a'_i-a_j'|\geq 2$  `ones' in the $i$th column to the $j$th column. So, we can construct $|a'_i-a_j'|\geq 2$ different loop-digraph realizations of $(a,b)$ from the one loop-digraph realization $G'$ of $(a',b)$ with such shifts. More formally, we define:
 
\begin{Definition}\label{Definition:ShiftLoopDigraph}
We call an operation on the $(n \times n)$-adjacency matrix $A'$ of loop-digraph realization $G'\in R_1(a',b),$ which switches the entries of $A'_{ki}=1$ and $A'_{kj}=0$ for one $k$, an \emph{$(i,j)$-shift on $G'$}. We denote the subset of loop-digraph realizations from $(a,b)$ which are constructed by $(i,j)$-shifts from a loop-digraph realization $(a',b)$ by $\textnormal{Shift}_{ij}(G',(a,b)).$
\end{Definition} 
 
 Let us now consider two arbitrary loop-digraph realizations of $(a',b),$ namely $G_1'$ and $G_2'$. Then it can happen that shifts in these two loop-digraphs produce an identical loop-digraphic realization, i.e.\ $\textnormal{Shift}_{i,j}(G_1',(a,b)) \cap \textnormal{Shift}_{i,j}(G_2',(a,b)) \neq \emptyset.$ We give an example.

\begin{Example}\label{Example:ShiftmengeLoopdigraphs}
We consider the adjacency matrices $A_1'$ and $A_2'$ of the two loop-digraph realizations $G_1'$ and $G_2'$ for loop-digraphic list $(a',b):=((4,1),(2,1),(0,1),(0,1),(0,1),(0,1))$ and apply on each of them all $4$ possible $(1,2)$-shifts. These shifts result in the loop-digraphic list $(a,b)=((3,1),(3,1),(0,1),(0,1),(0,1),(0,1)).$ Hence, $a'$ yields $a$ by a unit $(1,2)$-transfer. For

$A_1':=\left(\begin{matrix}
0&1&0&0&0&0&0&0\\
0&1&0&0&0&0&0&0\\
1&0&0&0&0&0&0&0\\
1&0&0&0&0&0&0&0\\
1&0&0&0&0&0&0&0\\
1&0&0&0&0&0&0&0\\
\end{matrix}\right)$ and 
$A_2':=\left(\begin{matrix}
0&1&0&0&0&0&0&0\\
1&0&0&0&0&0&0&0\\
0&1&0&0&0&0&0&0\\
1&0&0&0&0&0&0&0\\
1&0&0&0&0&0&0&0\\
1&0&0&0&0&0&0&0\\
\end{matrix}\right)$ we get 
$~$\\
$\textnormal{Shift}_{1,2}(G_1',(a',b)) \cap \textnormal{Shift}_{1,2}(G_2',(a,b))=\left\{\left(\begin{matrix}0&1&0&0&0&0&0&0\\0&1&0&0&0&0&0&0\\0&1&0&0&0&0&0&0\\1&0&0&0&0&0&0&0\\
1&0&0&0&0&0&0&0\\1&0&0&0&0&0&0&0\\\end{matrix}\right)\right\}.$ More generally there are ${6 \choose 4}$ possible loop-digraph realizations for $(a',b)$ and ${6 \choose 3}$ loop-digraph realizations for $(a,b).$
\end{Example}

Carefully comparing the two loop-digraphs of the adjacency matrices we observe that the symmetric difference of their arc sets contains exactly four arcs forming a directed alternating cycle, i.e. all four arcs alternate in their direction and in their appearance in $G_1'$ and $G_2'$. Especially, loop-arcs are possible. Moreover, all of these arcs correspond to the first or second column of the matrices -- the same columns as in the $(1,2)$-shift. In the following proposition we see that $(i,j)$-shifts on two loop-digraph realizations can lead to identical realizations if and only if the symmetric difference of their arc set is an alternating $4$-cycle and contains arcs which can be shifted. 

\begin{Proposition}\label{PropositionSymmetrischeDifferenz}
Let $(a,b)$ and $(a',b)$ be two different loop-digraphic lists such that $a'$ yields $a$ by a unit $(i,j)$-transfer. Furthermore, we assume that $G_1'$ and $G_2'$ with $G_1'\neq G_2'$ are loop-digraph realizations of $(a',b)$. $\textnormal{Shift}_{i,j}(G_1',(a,b)) \cap \textnormal{Shift}_{i,j}(G_2',(a,b)) \neq \emptyset$ if and only if the symmetric difference of their arc sets is a directed alternating cycle on the four vertices $i,j,k,k'$ with $k \neq k'$, i.e.\ $A(G_1')\Delta A(G_2')=\{(k,i),(k',j),(k,j),(k',i)\}$ with $(k,i),(k',j) \in A(G_1')\setminus A(G_2')$ and $(k,j),(k',i) \in A(G_2')\setminus A(G_1')$.
\end{Proposition}

\begin{proof}
First we consider a loop-digraphic realization $G$ with $G \in \textnormal{Shift}_{i,j}(G_1',(a,b)) \cap \textnormal{Shift}_{i,j}(G_2',(a,b))\subset R_1(a,b).$  Then there exist two different $(i,j)$-shifts -- one in $G_1'$ and one in $G_2',$ say a shift changing arc $(k',i) \in A(G_2')$ to arc $(k',j) \notin A(G_2')$ and arc $(k,i) \in A(G_1')$ to arc $(k,j) \notin A(G_1').$ Since both $(i,j)$-shifts lead to the loop-digraph realization $G$, we can conclude for arcs in $G_1'$ and $G_2'$ that $(k',i) \notin A(G_1'),$ $(k,i) \notin A(G_2'),$ $(k',j) \in A(G_1')$ and $(k,j) \in A(G_2').$ More differences between $G_1'$ and $G_2'$ cannot exist.\\
The converse implication holds trivially. For the two loop-digraph realizations $G_1'$ and $G_2'$ we define the $(i,j)$-shifts $(k',i) \in A(G_1')$ to $(k',j) \notin A(G_1')$, and $(k,i) \in A(G_2')$ to $(k,j) \notin A(G_2').$ Clearly, we get the same new realization $G.$
\end{proof}

We call two loop-digraph realizations $G_1',G_2'$ with a symmetric difference of their arc sets like in Proposition \ref{PropositionSymmetrischeDifferenz} \emph{$(i,j)$-adjacent}, i.e.\ their symmetric difference is a directed alternating cycle of length four. We define the subset of all loop-digraph realizations $R_1(a',b)$ of $(a',b)$ possesing at least one $(i,j)$-adjacent realization by $M_{i,j}(a',b):=$
$$\{G' \in R_1(a',b)|~\textnormal{it exists a }G_2' \in R_1(a',b)\textnormal{ such that }G' \textnormal{ and }G_2' \textnormal{ are $(i,j)$-adjacent.}\}$$
Applying $(i,j)$-shifts on several lists in $M_{i,j}(a',b)$ can lead to identical realizations of $(a',b)$ which was proven in Proposition~\ref{PropositionSymmetrischeDifferenz} and makes it difficult to estimate the number of created realizations. Therefore, we consider the remaining set of $R_1(a',b)$, where applying an $(i,j)$-shift on each pair of realizations produces a different pair of realizations in $R_1(a,b)$. More formally, we get for two loop-digraph realizations $G_1',G_2' \in R_1(a',b)\setminus M_{i,j}(a',b)$ that $\textnormal{Shift}_{i,j}(G_1',(a,b)) \cap \textnormal{Shift}_{i,j}(G_2',(a,b)) = \emptyset.$ Hence, the number of loop-digraph realizations which can be constructed by shifts from elements in $R_1(S')\setminus M_{i,j}(a',b)$ is at least $|a_i'-a_j'|\cdot |R_1(a',b)\setminus M_{i,j}(a',b)|\geq 2 \cdot |R_1(a',b)\setminus M_{i,j}(a',b)|.$ 

\begin{Proposition}\label{Proposition:NumberNonAdjacentloopRealizations}
Let $(a,b)$ and $(a',b)$ be two loop-digraphic lists such that $a'$ yields $a$ by a unit $(i,j)$-transfer. Applying all possible $(i,j)$-shifts on all loop-digraphic realizations in $R_1(a',b)\setminus M_{i,j}(a',b)$ we get a subset of $R_1(a,b)$ that has at least twice the cardinality of $R_1(a',b) \setminus M_{i,j}(a',b).$ In particular, we get
\begin{eqnarray*}
\left|\bigcup_{G' \in \left(R_1(a',b)\setminus M_{i,j}(a',b)\right)}\textnormal{Shift}_{i,j}(G',(a,b))\right| & \geq & |a_i'-a_j'|\cdot |R_1(a',b)\setminus M_{i,j}(a',b)|\\
&\geq & 2 \cdot |R_1(a',b)\setminus M_{i,j}(a',b)|.
\end{eqnarray*}
\end{Proposition}

Let us consider an extremal example for such a situation where we apply one after another $(k,k+1)$-shifts for $3 \leq k \leq n-1$ on loop-digraph realizations which are not $(k,k+1)$-adjacent. Clearly, this leads to exponentially many loop-digraph realizations.

\begin{Example}\label{Example:ExponentialManyLoopdigraphRealizations}
Let us consider threshold loop-digraphic list
$$(a',b):=((n-1,0),(n-2,1),(n-3,2),\dots,(2,n-3),(1,n-2),(0,n-1))$$ possessing exactly one loop-digraph realization. The loop-digraphic Ferrers matrix of $(a',b)$ only possesses entries $0$ above and on its main diagonal. Below this diagonal its enries are $1.$ \\
First we apply an $(1,3)$-shift and get two loop-digraph realizations of list $(a^2,b)$. Then we apply one after another a $(3,4)$-shift, $(4,5)$-shift,\dots,$(n-1,n)$-shift and get realizations for lists $(a^3,b),(a^4,b),\dots,(a^{n-1},b)=:(a,b)$. We get  $(a,b)=((n-2,0),(n-2,1),\dots,(1,n-2),(1,n-1)).$ In each step we have two possibilities for a shift. On the other hand $M_{k,k+1}((a^{k-1},b))= \emptyset,$ because realizations of $R_1(a^{k-1},b)$ can only differ in the first $k$ columns but not in the $(k+1)$th column. Hence, $N_1(a,b)\geq 2^{n-2}.$
\end{Example}

Note that $(a',b)$ in our example can yield $(a,b)$ by a unit $(1,n)$-transfer. That is, only one shift would have been sufficient to achieve loop-digraphic list $(a,b).$ In this case we only have $n-1$ possible $(1,n)$-shifts and so our estimation for a lower bound of $N_1(a,b)$ is $n-1.$ Hence, the kind of transfer paths plays an important role for the estimation of the possible lower bound for the number of realizations. In particular, the situation of the existence of two different $(i,j)$-adjacent loop-digraph realizations of a list $(a^k,b)$ on a transfer path $(a',b),(a^2,b),\dots, (a^r,b)$ ($(a',b)$ is threshold list) can only appear if there were $(i,c)$-shifts and $(d,j)$-shifts on transfer subpath $(a',b),\dots,(a^{k-1},b).$ Hence, we can conclude the following result.

\begin{corollary}\label{Corollar:exponentiellerTransferPath}
Let $(a,b)$ be a loop-digraphic list and $(a',b)$ its threshold loop-digraphic list. If there exists a transfer path $(a',b),\dots,(a,b)$ of length $r$ such that we have for each pair of transfers $t_{i,j}$ and $t_{c,d},$ $i\neq c$ and $j \neq d,$ then we have $N_1(a,b)\geq 2^r.$ 
\end{corollary}

Note that it is possible that we have $j=c$ or $i=d$ as in Example~\ref{Example:ExponentialManyLoopdigraphRealizations}. In the next steps, we need a combinatorial insight for binomial coefficients.

\begin{Proposition}\label{Proposition:PascalsTriangle}
Let $d:=c-\ell$ where $c,d,\ell \in \mathbb{N}.$ For $\ell\geq 1$ we have ${2c-\ell \choose c-\ell+1} \geq {2c-\ell \choose c}.$ For $\ell\geq 2$ we have ${2c-\ell \choose c-\ell+1} > {2c-\ell \choose c}.$
\end{Proposition}

\begin{proof}
We consider Pascal's triangle in row $2c-\ell.$ For an even $2c-\ell,$ we find the maximum binomial coefficient ${2c-\ell \choose c-\frac{\ell}{2}}.$ In this case $\ell$ must be even and therefore $\ell \geq 2$. Clearly, the binomial coefficient decreases symmetrically starting on the maximum middle binomial coefficient in the directions of both borders of Pascal's triangle. Since, $|c-(c-\frac{\ell}{2})|=\frac{\ell}{2}$ and $|c-\frac{\ell}{2}-(c-\ell+1)|=\frac{\ell}{2}-1$, bionomial coefficient ${2c-\ell \choose c-\ell+1}$ is nearer to the maximum bionomial coefficient than ${2c-\ell \choose c}$. Hence, ${2c-\ell \choose c-\ell+1} > {2c-\ell \choose c}.$ For an odd $2c-\ell,$ we find the two maximum binomial coefficients ${2c-\ell \choose c-\frac{1}{2}(\ell+1)}$ and ${2c-\ell \choose c-\frac{1}{2}(\ell-1)}$ in row $2c-\ell$ of Pascal's triangle. Again, the binomial coefficients decrease symmetrically starting on the two maximum middle binomial coefficients in the directions of both borders of Pascal's triangle. Since we have for $\ell\geq 3$,
  
\begin{enumerate}
\item $|c-(c-\frac{1}{2}(\ell+1))|=\frac{1}{2}(\ell+1)$,
\item $|c-\ell+1-(c-\frac{1}{2}(\ell+1))|=\frac{1}{2}(\ell-3)$
\item $|c-(c-\frac{1}{2}(\ell-1))|=\frac{1}{2}(\ell-1)$,
\item $|c-\ell+1-(c-\frac{1}{2}(\ell-1))|=\frac{1}{2}(\ell-1)$ and
\end{enumerate}

$\frac{1}{2}(\ell-3)<\frac{1}{2}(\ell-1)$ we get that ${2c-\ell \choose c-\ell+1}$ is nearer to the right maximum binomial coefficient than ${2c-\ell \choose c}$ to the left maximum binomial coefficient in Pascal's triangle. Hence, we have ${2c-\ell \choose c-\ell+1} > {2c-\ell \choose c}$ for $\ell\geq 2.$ Let us finally consider the case $\ell=1.$ We find that ${2c-\ell \choose c-\ell+1}={2c-\ell \choose c}.$ Hence, we have ${2c-\ell \choose c-\ell+1}\geq {2c-\ell \choose c}$ for $\ell\geq 1.$
\end{proof}

Let us now consider all loop-digraph realizations which are constructed by $(i,j)$-shifts on elements in $M_{i,j}(a',b).$ Then we find the following result.

\begin{Proposition}\label{Proposition:NumberAdjacentloopRealizations}
Let $(a,b)$ and $(a',b)$ be two different loop-digraphic lists such that $a'$ yields $a$ by a unit $(i,j)$-transfer and $M_{i,j}(a',b)\neq \emptyset.$ Applying all possible $(i,j)$-shifts on all elements in $M_{i,j}(a',b)\subset R_1(a',b)$ we get a subset of loop-digraph realizations $R_1(a,b)$ that has a larger cardinality than $M_{i,j}(a',b),$ i.e.\ \\
$\left|\bigcup _{G' \in  M_{i,j}(a',b)} \textnormal{Shift}_{i,j}(G',(a,b))\right| > |M_{i,j}(a',b)|.$ 
\end{Proposition}

\begin{proof}
Two $(i,j)$-adjacent loop-digraph realizations in $M_{i,j}(a',b)$ do only differ in the $i$th and $j$th columns. There are at least $|a_i'-a_j'|\geq 2$ and at most $a_i'$ possible $(i,j)$-shifts in each realization. More precisely, there exist $c$ $(i,j)$-shifts in each such realization with $2\leq |a_i'-a_j'|\leq c \leq a_i'.$ Hence, we can conclude for an adjacency matrix of such a realization that there exist $c\leq a_i'$ rows with entry one in column $i$ whereas the entries in the $j$ths column of these rows are zero. Consider a schematic picture where the permutation of the indices has been ignored. That means, the rows have been permuted to form four different blocks. $A':=$

\[
\begin{array}{cccccccccccccccccc}
          &           && 1 && \ldots && i             && \ldots && j             && \ldots && n                                      \\
1         & \LiKl{13} &&   &&        && \ob{1}        &&        && \ob{0}        &&        &&   & \ReKl{13} & \rdelim\}{4}{0pt}[$c$]\\
\bigdots  &           &&   &&        && \mi{1}        &&        && \mi{0}                                                            \\
          &           &&   &&        && \mi{\myvdots} &&        && \mi{\myvdots}                                                     \\
c         &           &&   &&        && \un{1}        &&        && \un{0}                                                            \\
\bigdots  &           &&   &&        && \ob{0}        &&        && \ob{1}        &&        &&   &           & \rdelim\}{3}{0pt}[$d$]\\
          &           &&   &&        && \mi{\myvdots} &&        && \mi{\myvdots}                                                     \\
c+d         &           &&   &&        && \un{0}        &&        && \un{1}                                                            \\
\bigdots  &           &&   &&        && \ob{1}        &&        && \ob{1}                                                            \\
          &           &&   &&        && \mi{\myvdots} &&        && \mi{\myvdots}                                                     \\
          &           &&   &&        && \un{1}        &&        && \un{1}                                                            \\
          &           &&   &&        && \ob{0}        &&        && \ob{0}                                                            \\
          &           &&   &&        && \mi{\myvdots} &&        && \mi{\myvdots}                                                     \\
n         &           &&   &&        && \un{0}        &&        && \un{0}                                                            \\
          &           &&   &&        && a_i'          &&        && a_j'
\end{array}
\]

Since $a'_i \geq a'_j +2,$ we find in the  $i$th column $d=c-\ell$ rows with entries "zero" ($\ell\geq 2$) whereas the entries of these rows in the $j$th column are "one". There are different possibilities for $c$ and $d,$ but for two $(i,j)$-adjacent elements in $M_{i,j}(a',b)$ these values are fixed with Proposition \ref{PropositionSymmetrischeDifferenz}. 
Note that we only consider matrices $A' \in M_{ij}(a',b)$ in a fixed scenario $F$, i.e.\ all entries that do not belong to the $i$th and $j$th column of a matrix in scenario $F$ are identical. Clearly, only two matrices in such a fixed scenario can be $(i,j)$-adjacent. All such scenarios we denote by $\mathcal{F}$. We have additionally $d\geq 1$ and $c\geq 3.$ Otherwise there is no possibility for finding an $(i,j)$-adjacency for two realizations in $M_{i,j}(a',b).$ Hence, for each fixed pair $c,d$ can either exist ${c+d \choose c}={2c-\ell \choose c}$ loop-digraph realizations of $(a',b)$ or no loop-digraph realization of $(a',b)$ in $M_{i,j}(a',b).$ This is true, because the entries in the corresponding $(c+d)$ rows in columns $i$ and $j$ can be permuted and maintain the row sums. Hence, we can divide the set $M_{i,j}(a',b)$ in disjoint subsets $m_F(c,d)$ for each fixed pair $(c,d)$ and each scenario $F$. That is 

$$M_{i,j}(a',b)=\bigcup _{\begin{matrix}
3\leq c \leq a_i'\\
1\leq d \leq  c-\ell\end{matrix}} \bigcup_{F \in \mathcal{F}}m_F(c,d).$$ 

It is sufficient to consider one arbitrary $m_F(c,d)$ which is not empty. Now, we can apply all possible $(i,j)$-shifts on each realization of $m_F(c,d).$ We get exactly ${c+d \choose d+1}={2c-l \choose c-\ell+1}$ different loop-digraph realizations of $(a,b).$ Since, ${2c-\ell \choose c-\ell+1} > {2c-\ell \choose c}$ for all non-empty $m_F(c,d)$ with $\ell\geq 2$ (Proposition~\ref{Proposition:PascalsTriangle}) our proof is done.
\end{proof}

We put the parts of all propositions together and get Theorem~\ref{CountloopDigraphRealizations}.

\begin{thm}\label{CountloopDigraphRealizations}
Let $(a,b)$ and $(a',b)$ be two different lists such that $a'$ yields $a$ by a unit $(i,j)$-transfer. Then it follows $N_1(a,b)\geq N_1(a',b).$ If $(a',b)$ is a loop-digraphic list, then we have  $N_1(a,b) > N_1(a',b).$
\end{thm}

\begin{proof}
 If $(a',b)$ is not a loop-digraphic list, then $N_1(a',b)=0$ and the inequality holds trivially. So let us consider the case that $(a',b)$ is a loop-digraphic list. Then $(a,b)$ is also a loop-digraphic list by Remark~\ref{Remark:TheoremLoopDigraphSequences}. A loop-digraph realization of $(a',b)$ is either in $M_{i,j}(a',b)$ or in $R_1(a',b)\setminus M_{i,j}(a',b)$. Now we apply for all realizations of $(a',b)$ all possible $(i,j)$-shifts. Then we get 
 
$$\bigcup_{G' \in R_1(a',b)}\textnormal{Shift}_{i,j}(G',(a,b))\subset R_1(a,b).$$
 
(Note that not all elements in $R_1(a,b)$ are necessarily achieved by such shifts.) We apply Propositions \ref{Proposition:NumberNonAdjacentloopRealizations} and \ref{Proposition:NumberAdjacentloopRealizations} and get 
\begin{eqnarray*}
N_1(a,b) &\geq &\left|\bigcup_{G' \in R_1(a',b)}\textnormal{Shift}_{i,j}(G',(a,b))\right|\\
&=&\left|\bigcup_{G' \in R_1(a',b)\setminus M_{i,j}(a',b)}\textnormal{Shift}_{i,j}(G',(a,b)))\right| + \left|\bigcup_{G' \in  M_{i,j}(S')}\textnormal{Shift}_{i,j}(G',(a,b))\right|\\
&\geq& 2|R_1(a',b)\setminus M_{i,j}(a',b)|+|M_{i,j}(a',b)|\\
&>&N_1(a',b).
\end{eqnarray*}
\end{proof}

If we now consider a transfer path between two loop-digraphic lists $(a,b)$ and $(a',b)$ where $a'$ majorizes $a,$ we can easily conclude by Theorem~ \ref{CountloopDigraphRealizations} the following general result with respect to majorization.

\begin{corollary}\label{Korollar:CountLoopDigraphRealizations}
Let $(a,b)$ and $(a',b)$ be two different loop-digraphic lists, vector $a$ is nonincreasing and $a \prec a'.$ Then $N_1(a,b) > N_1(a',b).$
\end{corollary}

\begin{proof}
By Theorem~\ref{Theorem:LoopDigraphSequences} there exists at least one transfer path $(a^1,b),\dots,(a^r,b)$ where $(a^1,b)=(a',b)$,  $(a^r,b):=(a,b)$ and $a^{i+1} \prec a^{i}$. We show by induction on $r$ the correctness of the claim. For $r=2$ we apply 
Theorem~\ref{CountloopDigraphRealizations} and get $N_1(a',b)<N_1(a,b).$ We consider the transfer path $(a^2,b),\dots, (a^r,b).$ With our induction hypothesis we can conclude $N_1(a^2,b)<N_1(a^r,b).$ For $(a^1,b)$ and $(a^2,b)$ we apply again 
Theorem~\ref{CountloopDigraphRealizations}. This yields $N_1(a^1,b)<N_1(a^2,b)<N_1(a^r,b)).$
\end{proof}

Consider again Example~\ref{Example:transferPaths}. Clearly, the loop-digraphic list $(a^*,b^*):=$\\
$((2,2),(2,2),(1,1),(1,1))$ possesses the most loop-digraph realizations in the set of all lists with $4$ pairs and $m=6.$ It is not the only loop-digraphic list with this property. If we consider lists $(a^*,b^*_{\tau})$, i.e.\ we permute $b^*$ with an arbitrary permutation $\tau,$ we get a loop-digraphic list that possesses the same number of realizations. This can easily be seen if we consider an equivalent formulation of the loop-digraph realization problem, namely the  bipartite realization problem. In each bipartite realization $G^*$ of list $(a^*,b^*)$ we have only to permute the indices of the vertices $v_i$ in the second independent vertex set corresponding to the current permutation $\tau$ of $b^*.$ Hence, the number of bipartite realizations is identical for each such permutation.

\begin{Proposition}\label{PropositionOrderingOfTuplesLoopDigraphSequences}
Let $(a,b)$ be a loop-digraphic list with nonincreasing $a$ and $(a,b_{\sigma})$ a list where vector $b$ was permuted by permutation $\sigma:\{1,\dots,n\}\mapsto \{1,\dots,n\}.$ Then $(a,b_{\sigma})$ is a loop-digraphic list and the number of loop-digraph realizations of $(a,b)$ and $(a,b_{\sigma})$ is identical.
\end{Proposition} 

\begin{proof}
We define an $(n \times n)$-\emph{permutation matrix} $P_{\sigma}$ with 
$P_{i\sigma(j)}:=\begin{cases}
1& \textnormal{if } i=j\\
0& \textnormal{if } i \neq j\\
\end{cases}.$
We consider for each loop-digraph realization of $(a,b)$ or $(a,b_{\sigma})$ the corresponding adjacency matrix $A$ or $A_{\sigma},$ respectively. Furthermore, we define a function $b:R_1(a,b_{\sigma})\mapsto R_1(a,b)$ with $b(A_{\sigma}):=P_{\tau}A_{\sigma}.$ Since, $b$ permutes the rows of $A_{\sigma}$ (and therefore the `outdegrees') we get the adjacency matrix of a loop-digraph realization of $(a,b).$ Since $b$ is bijective, the number of realizations of $(a,b)$ and $(a,b_{\sigma})$ is identical.
\end{proof}

For the digraph realization problem this statement is not true. We consider the details in the next subsections.
In a last step of this subsection we consider a special type of lists. First we define an integer list
$\alpha:=(\alpha_1,\dots,\alpha_n)$ for a constant $m\in \mathbb{N}$ with $m\leq n^2$ by
$$\alpha_i:=\begin{cases}
\lfloor \frac{m}{n}\rfloor +1& \textnormal{for } i \in \{1,\dots,m \bmod(n)\}\\
\lfloor \frac{m}{n}\rfloor& \textnormal{for } i \in \{m \bmod(n)+1,\dots,n\}\\
\end{cases}$$

Clearly, $\sum_{i=1}^{n}\alpha_i=m.$ If $n$ divides $m,$ then we have $\alpha_i=\frac{m}{n}$ for $1 \leq i \leq n.$ 

\begin{Definition}
Let $\sigma: \{1,\dots,|V|\} \mapsto \{1,\dots,|V|\}$ be an arbitrary permutation and $\alpha_{\sigma}$ a permutation of integer list $\alpha.$ We call a list $(\alpha,\alpha_{\sigma})$ \emph{minconvex list}. 
\end{Definition}

In our next theorem we show that each nonincreasing integer list $a$ with $\sum_{i=1}^{n}a_i=m$ majorizes $\alpha.$ With 
Theorem~\ref{Theorem:Hardy} by Polya, Hardy and Littlewood this implies that $\sum_{i=1}^{n}g(a_i) \geq \sum_{i}^{n}g(\alpha_i)$ for all integer lists $a$, where $g:\mathbb{Z}\mapsto \mathbb{Z}$ denotes an arbitrary convex function. This is the intuition behind the notion minconvex list.

\begin{thm}\label{Theorem:MinconvexSequence}
Let $a$ be a nonincreasing integer list with $\sum_{i=1}^{n}a_i=m.$ Then we have $\alpha \prec a.$
\end{thm}

\begin{proof}
Assume there exists $k <n$ with $\sum_{i=1}^{k}a_i < \sum_{i=1}^{k} \alpha_i.$ Let $k_0$ be the smallest of such $k$s. Since $\sum_{i=1}^{k_0-1}a_i \geq \sum_{i=1}^{k_0-1} \alpha_i$ for $k_0 >1$ it follows $\alpha_{k_{0}}>a_{k_{0}}\geq a_{k_{0}+1}\geq \dots \geq a_{n}.$ Then we get $\sum_{i=1}^{n}a_i <m$ in contradiction to our assumption.
\end{proof}

Now, we are able to prove a general result for minconvex lists.

\begin{corollary}\label{Korollar:NumberRealizationsMinconvexLoop}
Let $(a,b)$ be a loop-digraphic list with nonincreasing $a$ and $\sum_{i=1}^{n}a_i=m,$ which is not a minconvex list. Then we find for an arbitrary minconvex list $(\alpha,\alpha_{\sigma})$ that $N_1(\alpha,\alpha_{\sigma})>N_1(a,b).$
\end{corollary}

\begin{proof}
We transform list $(a,b)$ into the minconvex list $(\alpha,\alpha_{\tau})$ by the following series of lists; $(a,b),(\alpha,b),(\alpha_{\tau},b_{\tau}),(\alpha_{\tau},\alpha),(\alpha,\alpha_{\tau})$ where $\tau$ is a permutation such that $b_{\tau}$ is nonincreasing. Since by Theorem~\ref{Theorem:MinconvexSequence} $\alpha \prec a$ and $\alpha \prec b_{\tau}$ and $\alpha$ is nonincreasing by definition we get by Corollary~\ref{CountloopDigraphRealizations} that $N_1(a,b)< N_1(\alpha,b)$ and $N_1(\alpha_{\tau},b_{\tau})< N_1(\alpha_{\tau},\alpha)$ if we switch the components in the last pair of lists. Clearly, $N_1(\alpha_{\tau},b_{\tau})= N_1(\alpha,b)$ and $N_1(\alpha_{\tau},\alpha)= N_1(\alpha,\alpha_{\tau})$. In summery and if we take into account that at least one step from $(a,b)$ to $(\alpha,b)$ or from $(\alpha_{\tau},b_{\tau})$ to $(\alpha_{\tau},\alpha)$ is necessary when $(a,b)$ is not a minconvex list, we obtain $N_1(\alpha,\alpha_{\tau})> N_1(a,b).$ By Proposition~\ref{PropositionOrderingOfTuplesLoopDigraphSequences} we find that $N_1(\alpha,\alpha_{\tau})=N_1(\alpha,\alpha_{\sigma})$ and the proof is done.
\end{proof}

\subsection{The number of digraph realizations and majorization}

In this subsection we try to follow the proofs of the last subsection with respect to the digraph realization problem. We start with an example showing that we cannot apply all approaches analogously. It turns out that we have to modify the main result in Corollary \ref{Korollar:CountLoopDigraphRealizations}. In particular, it is necessary to sort a digraphic list $(a,b)$ in lexicographically nonincreasing order. We take the notions of $R_2(a,b)$ as the set of all digraph realizations for digraphic list $(a,b)$ and $N_2(a,b):=|R_2(a,b)|$. Furthermore, we define $(i,j)$-shifts. 

\begin{Definition}\label{Definition:ShiftDigraph}
We call an operation on the $(n \times n)$-adjacency matrix $A'$ of digraph realization $G'\in R_2(a',b)$ that switches the entries of $A'_{ki}=1$ and $A'_{kj}=0$, where $k\neq j$, an \emph{$(i,j)$-shift on $G'$}. We denote the subset of digraph realizations of $(a,b)$ which are constructed by $(i,j)$-shifts from a digraph realization $G'$ by $Shift_{ij}(G',(a,b)).$
\end{Definition} 

Let us start with an example to discuss some problems leading to the above mentioned restriction.

\begin{Example}\label{Example:DigraphCounterExample}
First we apply the unique $(3,4)$-shift on the adjacency matrix $A':=\left(\begin{matrix}
0&0&0&0\\
1&0&1&0\\
0&1&0&0\\
1&1&1&0\\
\end{matrix}\right)$ of the digraph threshold list $(a',b)=((2,0),(2,2),(2,1),(0,3))$ leading to the adjacency matrix $A=\left(\begin{matrix}
0&0&0&0\\
1&0&0&1\\
0&1&0&0\\
1&1&1&0\\
\end{matrix}\right)$ of another threshold list $(a,b)=((2,0),(2,2),(1,1),(1,3))$. It follows $N_2(a',b)=N_2(a,b)=1.$ On the other hand we have $a \prec a'.$ Hence, a strict analogous result of Theorem~\ref{CountloopDigraphRealizations} for digraphic lists is not possible.
\end{Example}

The reason for this observation is that it is only possible to apply a unique $(i,j)$-shift in contrast to 
Proposition~\ref{Proposition:NumberNonAdjacentloopRealizations} where always at least two shifts are possible. On the other hand it turns out (see Theorem~\ref{Theorem:CountDigraphRealizationsLexicographic}), that this result is true if digraphic list $(a,b)$ is lexicographically nonincreasing. However, we cannot transfer all ideas from the proofs of the loop-digraph realization problem. To see this, consider the following example.

\begin{Example}\label{Example:Digraphlexicographical}
Consider the adjacency matrix $A'=\left(\begin{matrix}
0&1&0\\
1&0&0\\
1&0&0\\
\end{matrix}\right)$ of threshold list $(a',b):=((2,1),(1,1),(0,1)).$ We apply the unique $(1,3)$-shift and get $A=\left(\begin{matrix}
0&1&0\\
0&0&1\\
1&0&0\\
\end{matrix}\right).$ Note that the digraphic list $(a,b)=((1,1),(1,1),(1,1))$ of $A$ is lexicographically nonincreasing. On the other hand there exists another different digraph realization $G^*$ of $(a,b)$, i.e. $A^*=\left(\begin{matrix}
0&0&1\\
1&0&0\\
0&1&0\\
\end{matrix}\right).$ Hence, we have $N_2(a,b)>N_2(a',b)$ but $A^*$ cannot be achieved by a shift.
\end{Example}

It is easy to transfer the concept of $(i,j)$-adjacent loop-digraph realizations $G_1',G_2'\in R_2(a',b)$ -- with some little restrictions to avoid loops -- to digraph realizations.

\begin{Proposition}\label{PropositionSymmetrischeDifferenzdigraphRealization}
Let $(a,b)$ and $(a',b)$ be two different digraphic lists such that $a'$ yields $a$ by a unit $(i,j)$-transfer. Furthermore, we assume that $G_1'$ and $G_2'$ are digraph realizations of $(a',b)$ with $G_1'\neq G_2'$. $\textnormal{Shift}(G_1',(a,b)) \cap \textnormal{Shift}(G_2',(a,b)) \neq \emptyset$ if and only if we have for the symmetric difference of the arc sets, $A(G_1')\Delta A(G_2')=\{(k,i),(k',j),$ $(k,j),(k',i)\}$ with $(k,i),(k',j) \in A(G_1')\setminus A(G_2')$ and $(k,j),(k',i) \in A(G_2')\setminus A(G_1')$ where $k \notin \{k',i,j\}$ and $k' \notin \{k,i,j\}.$
\end{Proposition} 

The proof can be done as in Proposition~\ref{PropositionSymmetrischeDifferenz}. We call two digraph realizations $G_1',G_2'$ such that one yields the other by a swap \emph{$(i,j)$-adjacent}, i.e.\ their symmetric difference is as in Proposition~\ref{PropositionSymmetrischeDifferenzdigraphRealization}. Furthermore, we also use the notion of $M_{i,j}(a',b)$ as the set of all digraph realizations in $R_2(a',b)$ which possess at least one $(i,j)$-adjacent digraph realization.
We prove the claim of this subsection in two steps. First we transfer the proofs of the last chapter showing the results in a restricted case. This approach is analogous to the loop-digraph realization problem. In a second step we prove the results for the stronger case. Here we have to change some approaches. Lastly, we consider minconvex lists for digraphic lists. In this case we again have to modify our results with respect to loop-digraphic lists. It turns out, that minconvex lists $(\alpha,\alpha_{\tau})$ possess the largest number of digraph realizations if $\alpha$ is nonincreasing and $\alpha_{\tau})$ nondecreasing.

\begin{Proposition}\label{Proposition:NumberNonAdjacentRealizations}
Let $(a,b)$ and $(a',b)$ be two different digraphic lists such that $a'$ yields $a$ by a unit $(i,j)$-transfer. Applying all possible $(i,j)$-shifts on all elements in $R_2(a',b)\setminus M_{i,j}(a',b)$ we get a subset of $R_2(a,b)$ which is larger or equals the cardinality of $R_2(a',b) \setminus M_{i,j}(a',b).$ In particular,
\begin{eqnarray*}
\left|\bigcup_{G' \in (R_2(a',b)\setminus M)}\textnormal{Shift}_{i,j}(G',(a,b))\right| & \geq & (|a_i'-a_j'|-1)\cdot |R_2(a',b)\setminus M_{i,j}(a',b)|\\
& \geq & |R_2(a',b)\setminus M_{i,j}(a',b)|.
\end{eqnarray*} 
$$$$
\end{Proposition}

\begin{proof}
Consider an $(n \times n)$-adjacency matrix $A'$ of a digraph realization $G' \in R_2(a',b).$ With our assumption we have at least $|a_i'-a_j'|$ entries $A'_{ki}=1$ with $A'_{kj}=0.$ Clearly, $k \neq i$ but it can happen that $k=j.$ In this case there exist at least $(|a_i'-a_j'|-1)$ $(i,j)$-shifts. With our definition of transfers we have $a_i'\geq a'_j+2.$ This proves our claim.
\end{proof}

\begin{Proposition}\label{Proposition:NumberAdjacentRealizations}
Let $(a,b)$ and $(a',b)$ be two different digraphic lists such that $a'$ yields $a$ by a unit $(i,j)$-transfer. Applying all possible $(i,j)$-shifts on all elements in $M_{i,j}(a',b)$ we get a subset of digraph realizations $R_2(a,b)$ that is larger or equals the cardinality of $M_{i,j}(a',b),$ i.e.\ $\left|\bigcup _{G' \in  M_{i,j}(a',b)} \textnormal{Shift}(G',(a,b))\right| \geq  |M_{i,j}(a',b)|.$ 
\end{Proposition}

\begin{proof}
Let $G' \in M_{i,j}(a',b).$ Let us consider a schematic picture for the adjacency matrix of $G'$, i.e.\ $A'=$

\[
\begin{array}{cccccccccccccccccc}
          &           && 1 && \ldots && i             && \ldots && j             && \ldots && n                                      \\
1         & \LiKl{15} &&   &&        && \ob{1}        &&        && \ob{0}        &&        &&   & \ReKl{15} & \rdelim\}{4}{0pt}[$c$]\\
\bigdots  &           &&   &&        && \mi{1}        &&        && \mi{0}                                                            \\
          &           &&   &&        && \mi{\myvdots} &&        && \mi{\myvdots}                                                     \\
c        &           &&   &&        && \un{1}        &&        && \un{0}                                                            \\
\bigdots  &           &&   &&        && \ob{0}        &&        && \ob{1}        &&        &&   &           & \rdelim\}{3}{0pt}[$d$]\\
          &           &&   &&        && \mi{\myvdots} &&        && \mi{\myvdots}                                                     \\
c+d     &           &&   &&        && \un{0}        &&        && \un{1}                                                            \\
i         &           &&   &&        && \ob{0}        &&        && \ob{a'_{ij}}                                                      \\
j         &           &&   &&        && \un{a'_{ji}}  &&        && \un{0}                                                            \\
\bigdots  &           &&   &&        && \ob{1}        &&        && \ob{1}                                                            \\
          &           &&   &&        && \mi{\myvdots} &&        && \mi{\myvdots}                                                     \\
          &           &&   &&        && \un{1}        &&        && \un{1}                                                            \\
          &           &&   &&        && \ob{0}        &&        && \ob{0}                                                            \\
          &           &&   &&        && \mi{\myvdots} &&        && \mi{\myvdots}                                                     \\
n         &           &&   &&        && \un{0}        &&        && \un{0}                                                            \\
          &           &&   &&        && a_i'          &&        && a_j'
\end{array}
\]

There exist $c$ $(i,j)$-shifts in $G'$. Since $a'_i \geq a'_j +2,$ we find in the  $i$th column $d=c-\ell$ rows with entries "zero" ($\ell\geq 1$) whereas the entries of these rows in the $j$th column are "one". There are different possibilities for $c$ and $d,$ but for two $(i,j)$-adjacent elements in $M_{i,j}(a',b)$ these values are fixed with Proposition \ref{PropositionSymmetrischeDifferenzdigraphRealization}. Note that $d\geq 1,$ otherwise $G' \notin M_{i,j}(a',b).$ Note that we only consider matrices $A' \in M_{ij}(a',b)$ in a fixed scenario $F$, i.e.\ all entries that do not belong to the $i$th and $j$th column of a matrix in scenario $F$ are identical. Clearly, only two matrices in such a fixed scenario can be $(i,j)$-adjacent. All such scenarios we denote by $\mathcal{F}$. We distinguish in scenario $F$ between four different cases.

\begin{description}
\item[case 1:] $a'_{ij}=0$ and $a'_{ji}=1$. Then we have with $d:=c-\ell$ at least $2$ possible $(i,j)$-shifts. It follows $\ell\geq 1.$
\item[case 2:] $a'_{ij}=0$ and $a'_{ji}=0$. Then we have with $d:=c-\ell$ at least $3$ possible $(i,j)$-shifts. It follows $\ell \geq 2.$
\item[case 3:] $a'_{ij}=1$ and $a'_{ji}=1$. Then we have with $d:=c-\ell$ at least $3$ possible $(i,j)$-shifts. It follows $\ell \geq 2.$
\item[case 4:] $a'_{ij}=1$ and $a'_{ji}=0$. Then we have with $d:=c-\ell$ at least $4$ possible $(i,j)$-shifts. It follows $\ell \geq 3.$
\end{description}

In all four cases there do exist ${c+d\choose c}={2c-\ell \choose c}$ digraph realizations or no digraph realization of $(a',b)$. The reason is that the entries of the $(c+d)$ rows can be permuted maintaining the row and column sums. After all possible $(i,j)$-shifts in each non-empty case for a fixed pair $c,d$ in a scenario $F$, we get ${c+d \choose d+1}={2c-\ell \choose c-\ell+1}$ digraph realizations for each case. Applying 
Proposition~\ref{Proposition:PascalsTriangle} we obtain our claim. Note that equality can only appear in case~1. 
\end{proof}

\begin{thm}\label{Theorem:CountDigraphRealizations}
Let $(a,b)$ and $(a',b)$ be two different lists such that $a'$ yields $a$ by a unit $(i,j)$-transfer. Then it follows $N_2(a,b)\geq N_2(a',b).$ 
\end{thm}

\begin{proof}
If $(a',b)$ is not a digraphic list, then $N_2(a',b)=0$ and the inequality holds trivially. So let us consider the case that $(a',b)$ is a digraphic list. Then $(a,b)$ is also a digraphic list with Remark \ref{Remark:TheoremDigraphSequences}. A digraph realizations of $(a',b)$ is either in $M_{i,j}(a',b)$ or in $R_2(a',b)\setminus M_{i,j}(a',b)$. Now we apply for all these realizations of $(a',b)$ all possible $(i,j)$-shifts. Then we get $$\bigcup_{G' \in R_2(a',b)}\textnormal{Shift}_{i,j}(G',(a,b))\subset R_2(a,b).$$
  
(Note that not all elements in $R_2(a,b)$ are necessarily achieved by such shifts.) We apply Propositions \ref{Proposition:NumberNonAdjacentRealizations} and \ref{Proposition:NumberAdjacentRealizations} and get
 \begin{eqnarray*}
 & N_2(a,b) &\\
 & \geq & \left|\bigcup_{G' \in R_2(a',b)}\textnormal{Shift}_{i,j}(G',(a,b))\right|\\
&=&  \left|\left(\bigcup_{G' \in R_2(a',b)\setminus M_{i,j}(a',b)}\textnormal{Shift}_{i,j}(G',(a,b))\right) \cup  \left (\bigcup_{G' \in  M_{i,j}(a',b)}\textnormal{Shift}_{i,j}(G',(a,b))\right)\right|  \\
  &\geq & \left|R_2(a',b)\setminus M_{i,j}(a',b)\right|+\left|M_{i,j}(a',b)\right| \\
&=&N_2(a',b).
 \end{eqnarray*}
\end{proof}

\begin{corollary}\label{Corollary:NumberDigraphRealizationsNonincreasing}
Let $(a,b)$ and $(a',b)$ be two different digraphic lists, vector $a$ is nonincreasing and $a \prec a'.$ Then $N_2(a,b) \geq N_2(a',b).$
\end{corollary}

\begin{proof}
By Theorem~\ref{Theorem:DigraphSequences} there exists at least one transfer path $(a^1,b),\dots,(a^r,b)$ where $(a^1,b)=(a',b)$,  $(a^r,b):=(a,b)$ and $a^{i+1} \prec a^{i}$. We show by induction on $r$ the correctness of the claim. For $r=2$ we apply 
Theorem~\ref{Theorem:CountDigraphRealizations} and get $N_2(a',b) \leq N_2(a,b).$ We consider the transfer path $(a^2,b),\dots, (a^r,b).$ With our induction hypothesis we can conclude $N_2(a^2,b) \leq N_2(a^r,b).$ For $(a^1,b)$ and $(a^2,b)$ we apply again 
Theorem~\ref{Theorem:CountDigraphRealizations}. This yields $N_2(a^1,b)\leq N_2(a^2,b)\leq N_2(a^r,b).$
\end{proof}

\begin{thm}\label{Theorem:CountDigraphRealizationsLexicographic}
Let $(a,b)$ be a lexicographically nonincreasing list and $(a',b)$ a list such that $a'$ yields $a$ by a unit $(i,j)$-transfer. Then it follows $N_2(a,b)\geq N_2(a',b).$ If $(a',b)$ is a digraphic list, then we have  $N_2(a,b) > N_2(a',b).$
\end{thm}

\begin{proof}
If $(a',b)$ is not a digraphic list, then $N_2(a',b)=0$ and the inequality holds trivially. So let us consider the case that $(a',b)$ is a digraphic list. Then $(a,b)$ is also a digraphic list with Remark~\ref{Remark:TheoremDigraphSequences}. We distinguish between the cases, that $G' \in M_{i,j}(a',b)$ and $G' \in R_2(a',b)\setminus M_{i,j}(a',b).$  For $G' \in R_2(a',b)\setminus M_{i,j}(a',b)$  and  $a'_i \geq a'_j +3$ we always find at least two possible $(i,j)$-shifts for each such digraph realization. For $G' \in M_{i,j}(a',b)$ and $a'_i \geq a'_j+3$ consider all four cases in the proof of Proposition~\ref{Proposition:NumberAdjacentRealizations}. Clearly, we find there $\ell\geq 2$ in case 1. With Proposition \ref{Proposition:PascalsTriangle} we can prove our claim.\\
Let us now assume $a'_i:=a'_j+2.$ We first consider a digraph realization $G' \in M_{i,j}(a',b).$ Since $d>0$ and $a'_i:=a'_j+2$, we have $c \geq 2$ and can use an analogous proof as in Proposition~\ref{Proposition:NumberAdjacentRealizations}. We assume $d=0$ and consider a schematic picture for an adjacency matrix $A'$ with $A':=$

\[
\begin{array}{cccccccccccccccccc}
          &           && 1 && \ldots && i             && \ldots && j             && \ldots && n                                      \\
1         & \LiKl{12} &&   &&        && \ob{1}        &&        && \ob{0}        &&        &&   & \ReKl{12} & \rdelim\}{4}{0pt}[$c$]\\
\bigdots  &           &&   &&        && \mi{1}        &&        && \mi{0}                                                            \\
          &           &&   &&        && \mi{\myvdots} &&        && \mi{\myvdots}                                                     \\
c        &           &&   &&        && \un{1}        &&        && \un{0}                                                            \\
i         &           &&   &&        && \ob{0}        &&        && \ob{a'_{ij}}                                                      \\
j         &           &&   &&        && \un{a'_{ji}}  &&        && \un{0}                                                            \\
\bigdots  &           &&   &&        && \ob{1}        &&        && \ob{1}                                                            \\
          &           &&   &&        && \mi{\myvdots} &&        && \mi{\myvdots}                                                     \\
          &           &&   &&        && \un{1}        &&        && \un{1}                                                            \\
          &           &&   &&        && \ob{0}        &&        && \ob{0}                                                            \\
          &           &&   &&        && \mi{\myvdots} &&        && \mi{\myvdots}                                                     \\
n         &           &&   &&        && \un{0}        &&        && \un{0}                                                            \\
          &           &&   &&        && a_i'          &&        && a_j'
\end{array}.
\]

Note that we only consider matrices $A' \in M_{ij}(a',b)$ in a fixed scenario $F$, i.e.\ all entries that do not belong to the $i$th and $j$th column of a matrix in scenario $F$ are identical. Clearly, only two matrices in such a fixed scenario can be $(i,j)$-adjacent. All such scenarios we denote by $\mathcal{F}$. We distinguish in a scenario $F$ between four different cases.

\begin{description}
\item[case 1:] $a'_{ij}=0$ and $a'_{ji}=1$. Then we have $c=1$ and so one possible $(i,j)$-shift.
\item[case 2:] $a'_{ij}=0$ and $a'_{ji}=0$. Then we have $c=2$ and so two possible $(i,j)$-shifts.
\item[case 3:] $a'_{ij}=1$ and $a'_{ji}=1$. Then we have $c=2$ and so two possible $(i,j)$-shifts.
\item[case 4:] $a'_{ij}=1$ and $a'_{ji}=0$. Then we have with $c=3$ three possible $(i,j)$-shifts.
\end{description}

Note that $G' \in R_2(a',b)\setminus M_{i,j}(a',b)$ in all four cases, because one cannot find a valid directed alternating cycle of length four in $A'$. (That means that an existing alternating cycle cannot been switched to yield a new digraph realization.) We observe that case~1 and case~4 are mutually exclusive in a fixed scenario $F$. The reason is that we cannot yield $b_i$ and $b_j$ without changing entries in columns that are different from $i$ and $j$.\\
Let us assume that there exists at least one scenario $F$ with case~1; otherwise we get $N_2(a,b)>N_2(a',b)$. Then arc $(i,j)$ does not belong to any digraph realization of $(a',b).$ Since, $a_i=a_j$ and $(a,b)$ is lexicographically nonincreasing, we have $b_i \geq b_j.$ Since, $a'_{ji}=1,$ there exists in an arbitrary digraph realization $G'$ a $k \neq i,j$ with $a'_{ik}=1$ and $a'_{jk}=0.$ On the other hand we have our unique $(i,j)$-shift and so for a $\kappa\neq i,j$, $a'_{\kappa i}=1$ and $a'_{\kappa j}=0.$ Note that we can have $k=\kappa.$ After applying the $(i,j)$-shift in $G'$ we get a digraph realization $G \in R_1(a,b)$ containing a directed $3$-path $p:=(\kappa,j,i,k).$ We get a further digraph realization $G^*$ in $R_1(a,b)$ if we change path $p$ to $p^*:=(\kappa,i,j,k).$ This is possible, since arcs $(\kappa,i)$ and $(i,j)$ and $(j,k)$ are not in $G.$ The adjacency matrix $A^*$ of Digraph $G^*$ cannot been yield by an adjacency matrix $A'$ of case~4 of a scenario $F^* \neq F$, although we find $A^*_{ij}=1$ and $A^*_{ji}=0$ and the $k$th column was changed. The reason is that $d=0$ in $A^*$. Hence, $A^*$ cannot be constructed by an $(i,j)$-shift of any $A'$ Hence, we can find for each scenario $F$ with case~1 twice the amount of realizations. In summery, we get $N_2(a',b)<N_2(a,b).$
\end{proof}

Back to our Example~\ref{Example:Digraphlexicographical} we find with Theorem~\ref{Theorem:CountDigraphRealizationsLexicographic} that $N_2(a',b)<N_2(a,b).$ Here the directed path $p$ in our proof is a directed cycle of length three. In Example \ref{Example:DigraphCounterExample} we have to sort digraphic list $(a,b)$ in lexicographical order. That is $(a_{\sigma},b_{\sigma})=((2,2),(2,0),(1,3),$ $(1,1)).$ Hence, the adjacency matrix is its Ferrers matrix $A:=\left(\begin{matrix}
0&1&1&0\\
0&0&0&0\\
1&1&0&1\\
1&0&0&0
\end{matrix}\right).$ If we sort $(a',b)$ by permutation $\sigma$, we get $(a'_{\sigma},b_{\sigma})=((2,2),(2,0),(0,3),(2,1)).$ Hence we have $a'_{\sigma} \prec a_{\sigma},$ but $(a'_{\sigma},b_{\sigma})$ is not lexicographically nonincreasing. Hence, it is not possible to apply 
Theorem~\ref{Theorem:CountDigraphRealizationsLexicographic}.

\begin{corollary}\label{Korollar:CountLoopRealizations}
Let $(a,b)$ be a lexicographically nonincreasing digraphic list, $(a',b)$ be a digraphic list, $a \neq a'$ and $a \prec a'.$ Then $N_2(a,b) > N_2(a',b).$
\end{corollary}

\begin{proof}
There exists at least one transfer path $(a^1,b),\dots,(a^r,b)$ where $(a^1,b):=(a',b)$ and $(a^r,b):=(a,b)$ such that $a^i$ yields $a^{i+1}$ by a unit transfer see Theorem~\ref{Theorem:DigraphSequences}. We show by induction on $r$ the correctness of the claim. For $r=2$ we apply Theorem~\ref{Theorem:CountDigraphRealizationsLexicographic} and get $N_2(a',b)<N_2(a,b).$ Let us now assume $r>2.$ We consider the transfer path $(a^2,b),\dots, (a,b).$ With our induction hypothesis we can conclude $N_2(a^2,b)<N_2(a,b).$ For $(a',b)$ and $(a^2,b)$ we apply Theorem~\ref{Theorem:CountDigraphRealizations}. This yields $N_2(a',b)\leq N_2(a^2,b)<N_2(a,b).$
\end{proof}

\subsection{The number of digraph realizations for permuted lists}

In the next step we want to show that a digraphic list $(a,b)$ possesing a so-called opposed ordering has at least as many realizations as a list $(a_{\sigma},b)$ where $b$ was permuted by $\sigma$. This leads to the result that opposed minconvex lists possess the largest number of digraph realizations under all digraph realizations with $n$ pairs and fixed $m=\sum_{i=1}^{n}a_i.$ We start with a definition of these lists.

\begin{Definition}\label{Definition:opposedSequences}
Let $(a,b)$ be a digraphic list with nonincreasing $a$ and nondecreasing $b.$ We call $(a,b)$ an \emph{opposed list}.
\end{Definition}

In the next proposition we want to prove that a digraphic list $(a_{\tau},b)$ has more digraphic realizations as $(a,b)$, when the first list is `nearby' to the property to be opposed than $(a,b)$. It is not possible to apply the results of our last sections, neither Muirheads Lemma nor the majorization results. Clearly, a vector $a$ can be majorized by its permuted vector $a_{\tau}$, but it turns out that the nonincreasing permutation $a_{\tau}$ majorizes all other permutations of $a$. For Muirheads Theorem and for all of our majorization results we need the situation that the majorized vector is the nonincreasing one. This situation is not possible for permutations. It turns out that we find an contrary result of  Theorem~\ref{Theorem:DigraphSequences}. To prove this result we choose a similar procedure as in the beginning of this paper. We define left-transfers on lists and the corresponding operation left-shifts on matrices similar to transfers and shifts as in the last subsection. The difference is that we shift from right to left and two entries of a list (two column sums) can differ by one, which is not possible for transfers.

\begin{Definition}[left-transfer]
For an integer list $a$ with $a_i \leq a_j+1$ for $i,j$ such that $1 \leq i<j \leq n$, the list obtained from $a$ by an {\it (i,j)-left-transfer}  is the list $a+e_i+e_i$. Sometimes, we simply use the term \emph{left-transfer} without specifying the indices.
\end{Definition} 

\begin{Definition}\label{Definition:Left}
We call an operation on the $(n \times n)$-adjacency matrix $A'$ of digraph realization $G'\in R_2(a',b)$ that switches the entries of $A'_{ki}=0$ and $A'_{kj}=1$, where $k\neq j$ and $i<j$, an \emph{$(i,j)$-left-shift on $G'$}. 
\end{Definition}

\begin{Proposition}\label{PropositionDigraphSequencesLefttransfer}
Let $(a',b)$ be a digraphic list and let $(a,b)$ be a list where $b$ is nondecreasing and $a'$ yields $a$ by a unit left-transfer. Then $(a,b)$ is a digraphic list and $N_2(a,b)\geq N_2(a',b).$
\end{Proposition} 

\begin{proof}
There exist $i,j$ with $i<j$ such that $a'_{i}=a'_j-1$ and $a_i=a_j+1.$ Let $A'$ be an adjacency matrix of a digraph realization $G'$ of $(a',b)$. An $(i,j)$-left-shift on $A'$ results in an adjacency matrix $A$ of a digraph realization from $(a,b)$. We consider a schematic picture of such an $A'$ to see what types of $(i,j)$-left-shifts are possible, i.e.\ $A':=$  
\[
\begin{array}{cccccccccccccccccc}
          &           && 1 && \ldots && i             && \ldots && j             && \ldots && n                                      \\
1         & \LiKl{15} &&   &&        && \ob{1}        &&        && \ob{0}        &&        &&   & \ReKl{15} & \rdelim\}{4}{0pt}[$d$]\\
\bigdots  &           &&   &&        && \mi{1}        &&        && \mi{0}                                                            \\
          &           &&   &&        && \mi{\myvdots} &&        && \mi{\myvdots}                                                     \\
d        &           &&   &&        && \un{1}        &&        && \un{0}                                                            \\
\bigdots  &           &&   &&        && \ob{0}        &&        && \ob{1}        &&        &&   &           & \rdelim\}{3}{0pt}[$c$]\\
          &           &&   &&        && \mi{\myvdots} &&        && \mi{\myvdots}                                                     \\
d+c     &           &&   &&        && \un{0}        &&        && \un{1}                                                            \\
i         &           &&   &&        && \ob{0}        &&        && \ob{a'_{ij}}                                                      \\
j         &           &&   &&        && \un{a'_{ji}}  &&        && \un{0}                                                            \\
\bigdots  &           &&   &&        && \ob{1}        &&        && \ob{1}                                                            \\
          &           &&   &&        && \mi{\myvdots} &&        && \mi{\myvdots}                                                     \\
          &           &&   &&        && \un{1}        &&        && \un{1}                                                            \\
          &           &&   &&        && \ob{0}        &&        && \ob{0}                                                            \\
          &           &&   &&        && \mi{\myvdots} &&        && \mi{\myvdots}                                                     \\
n         &           &&   &&        && \un{0}        &&        && \un{0}                                                            \\
          &           &&   &&        && a_i'          &&        && a_j'
\end{array}
\]

There exist $c$ $(i,j)$-left-shifts in $A'$. Since $a'_i=a'_j -1,$ we find in the  $i$th column $c=d-\ell$ rows with entries "zero" ($\ell\geq 0$) whereas the entries of these rows in the $j$th column are "one". Note that we only consider matrices $A'$ from realizations in $R_2(a',b)$ for a fixed scenario $F$, i.e.\ all entries that do not belong to the $i$th and $j$th column of a matrix in scenario $F$ are identical. Clearly, only two matrices in such a fixed scenario can produce the same matrix $A$ after applying an $(i,j)$-left-shift. Additionally we need then $c\geq 1$. All such scenarios we denote by $\mathcal{F}$. We distinguish in scenario $F$ between four different cases.

\begin{description}
\item[case 1:] $a'_{ij}=0$ and $a'_{ji}=1$. Then we have $c:=d+2$ possible $(i,j)$-left-shifts.
\item[case 2:] $a'_{ij}=0$ and $a'_{ji}=0$. Then we have $c:=d+1$ possible $(i,j)$-left-shifts.
\item[case 3:] $a'_{ij}=1$ and $a'_{ji}=1$. Then we have $c=d+1$  possible $(i,j)$-left-shifts.
\item[case 4:] $a'_{ij}=1$ and $a'_{ji}=0$. Then we have $c:=d$ possible $(i,j)$-left-shifts. 
\end{description}

In all four cases there do exist ${c+d\choose c}$ digraph realizations or no digraph realization of $(a',b)$ for a fixed scenario $F$. The reason is that the entries of the $(c+d)$ rows can be permuted maintaining the row and column sums. Note that case~1 and case~4 are mutually exclusive in a fixed scenario $F$; the reason are the $i$th and $j$th row sums which can only maintained by changing the scenario. Let us first assume that case~1 exists in each scenario $F \in \mathcal{F}$. We consider a schematic picture of $A$ after applying an $(i,j)$-left-shift, i.e.\ $A:=$

\[
\begin{array}{cccccccccccccccccc}
          &           && 1 && \ldots && i             && \ldots && j             && \ldots && n                                      \\
1         & \LiKl{15} &&   &&        && \ob{1}        &&        && \ob{0}        &&        &&   & \ReKl{15} & \rdelim\}{4}{0pt}[$d+1$]\\
\bigdots  &           &&   &&        && \mi{1}        &&        && \mi{0}                                                            \\
          &           &&   &&        && \mi{\myvdots} &&        && \mi{\myvdots}                                                     \\
d        &           &&   &&        && \un{1}        &&        && \un{0}                                                            \\
\bigdots  &           &&   &&        && \ob{0}        &&        && \ob{1}        &&        &&   &           & \rdelim\}{3}{0pt}[$c-1$]\\
          &           &&   &&        && \mi{\myvdots} &&        && \mi{\myvdots}                                                     \\
d+c     &           &&   &&        && \un{0}        &&        && \un{1}                                                            \\
i         &           &&   &&        && \ob{0}        &&        && \ob{a'_{ij}}                                                      \\
j         &           &&   &&        && \un{a'_{ji}}  &&        && \un{0}                                                            \\
\bigdots  &           &&   &&        && \ob{1}        &&        && \ob{1}                                                            \\
          &           &&   &&        && \mi{\myvdots} &&        && \mi{\myvdots}                                                     \\
          &           &&   &&        && \un{1}        &&        && \un{1}                                                            \\
          &           &&   &&        && \ob{0}        &&        && \ob{0}                                                            \\
          &           &&   &&        && \mi{\myvdots} &&        && \mi{\myvdots}                                                     \\
n         &           &&   &&        && \un{0}        &&        && \un{0}                                                            \\
          &           &&   &&        && a_i'          &&        && a_j'
\end{array}
\]

Then we get in each case ${c+d \choose c-1}$ digraph realizations. In case~1 we have in scenario $F$, ${2d+2 \choose d+2}$ matrices of form $A'$ and ${2d+2 \choose d+1}$ matrices of form $A$, that is the number of realizations of $(a,b)$ becomes larger. In case~2 and case~3 we find in scenario $F$ ${2d+1 \choose d+1}$ matrices of form $A'$ and ${2d+1 \choose d}$ matrices of form $A.$ Hence, the number of realizations of $(a',b)$ equals the number of realizations in $(a,b)$ for these two cases. This connections can easily be seen in Pascals triangle. We get $N_2(a,b)\geq N_2(a',b).$\\
Let us now assume there exists a scenario $F$ with case~4. Remember that case~1 is absent in $F$. Since $b$ is nondecreasing there exists in scenario $F$ a $k$ with $k\neq i,$ $A'_{ik}=0$ and $A'_{jk}=1$;otherwise $b_{i}\leq b_j$ is not possible in case~4. Assume now that $d>0$. Then after applying an $(i,j)$-left-shift on each $A'$ in case~4 we get ${2c \choose c-1}$ realizations of form $A.$ If we apply the $(i,j)$-left-shift on entry $A'_{\kappa,j}=1$ we find in the corresponding realization $A$ a directed path $p=(\kappa,i,j,k)$ with $\kappa \neq i,j$. (Note that $\kappa =k$ is possible.) We reorient $p$ to $p^*=(\kappa,j,i,k)$ and get a new realization $G^*$ for $(a',b)$. Note that $G^*$ belongs to a different scenario $F^*$, since arc $(j,k)$ is here absent whereas it belongs to scenario $F$. Moreover, $G^*$ can be constructed from an $A'$ of case~1 in $F^*$. This is true, because in $A'$ of scenario $F$ exists a $\kappa' \neq \kappa$ with $A'_{\kappa' i}=1$ and $A'_{\kappa' j}=0$; otherwise $c>d$. But then we get $A'_{\kappa' i}=A_{\kappa' i}=A^*_{\kappa' i}=1$, $A'_{\kappa' j}=A_{\kappa' j}=A^*_{\kappa' j}=0$ and $A^*$ was constructed by an $(i,j)$-left-shift on $A^{'}$ in scenario $F^*$ where $A^{'}_{\kappa' i}=0$ and $A^{'}_{\kappa' j}=1$. Hence, we can find for each scenario $F$ with case~4 a further scenario $F^*$ with case~1. Let us know consider the number of realizations of such a pair of scenarios before and after an $(i,j)$-left-shift. It is sufficient to consider case~1 and case~4, because the other two cases remain the number of realizations after an $(i,j)$-left-shift. Case~4 has in scenario $F$ ${2c \choose c}$ matrices of form $A'$ and case~1 has in scenario $F^*$ ${2c \choose c+1}$ matrices of form $A'$. After applying all possible $(i,j)$-left-shifts, case~4 yields in scenario $F$ ${2c \choose c-1}$ matrices of form $A$ and case~1 produces in scenario $F^*$ ${2c \choose c}$ matrices of form $A$. Since ${2c \choose c+1}={2c \choose c-1}$ the number of digraph realizations for a pair of scenarios $F$ and $F^*$ equals the number of digraph realizations after all possible $(i,j)$-left-shifts. \\
It remains to consider a scenario $F$ with case~4 and $d=0$. Note that there can only exist one matrix $A'$ in $F$. Moreover, there is no possible  $(i,j)$-left-shift in $A'$. Since $A'_{ik}=0$ and $A'_{jk}=1$ (see above) we can find a digraph realization of $(a,b)$ in setting $A_{ik}=1$, $A_{jk}=0$, $A_{ij}=0$, $A_{ji}=1$ and for all other entries we take the entries of $A'$. Note that this is not a realization which was produced by another scenario $F^*$, because $c=d=0$ and so no left-shift can produce this realization. Hence, the unique realization $A'$ produces the unique realization $A$. In summery, we find for all cases that $N_2(a,b)\geq N_2(a',b).$
\end{proof}

\begin{Proposition}\label{PropositionOrderingOfTuplesDigraphSequences}
Let $(a,b)$ be a digraphic list and let $(a_{\tau},b)$ be a list where $b$ is nondecreasing and $a$ was permuted by a transposition $\tau:\{1,\dots,n\}\mapsto \{1,\dots,n\}$ such that there exist $i<j$ with $a_{i}<a_j$ and $a_{\tau(i)}>a_{\tau(j)}.$ Then $(a,b_{\tau})$ is a digraphic list and $N_2(a_{\tau},b)\geq N_2(a,b).$
\end{Proposition} 
 
\begin{proof}
We define $r:=a_j-a_i+1$ digraphic lists $(a^1,b),\dots,(a^r,b)$ such that $(a,b)=(a^1,b)$, $(a_{\tau},b)=(a^r,b)$ and each $a^k$ yields $a^{k+1}$ by a unit $(i,j)$-left-transfer. With Proposition~\ref{PropositionDigraphSequencesLefttransfer} each list $(a^k,b)$ is digraphic for $1 \leq k \leq r$ and we find $N_2(a_{\tau},b)\geq N_2(a,b).$
\end{proof}
 
Each list $(a,b)$ with nonincreasing $b$ can be changed by a list of transpositions $\tau_1,\dots,\tau_r$ in its opposed list $(a_{\sigma},b)$, such that the for a switch of the $i$th and $j$th component of $a$ with transposition $\tau_k$ where $i<j$, we have $a_{i}<a_j$ and $a_{\tau_k(i)}>a_{\tau_k(j)}.$ We apply the last proposition for each transposition step and obtain with $\sigma:=\tau_r \circ \tau_{r-1} \circ \dots \tau_1$ the following result.

\begin{thm}\label{Theorem:opposedSequences}
 Let $(a,b)$ be a digraphic list and $(a_{\sigma},b)$ be its opposed list, i.e.\ vector $b$ is nondecreasing and $\sigma$ is a permutation such that $a_{\sigma}$ is nonincreasing. Then $(a_{\sigma},b)$ is a digraphic list and $N_2(a_{\sigma},b)\geq N_2(a,b).$
 \end{thm}

Consider two simple examples. $(a,b):=((2,2),(2,1),(1,2))$ and its opposed list $(a,b_{\tau}):=((2,1),(2,2),(1,2))$ are both threshold lists and each possesses exactly one digraph realization. List $(a,b):=((2,1),(1,0),(0,2))$ is not digraphic but its opposed list $(a,b_{\tau}):=((2,0),(1,1),(0,2))$ is a threshold list. Hence we have $N_2(a,b)<N_2(a,b_{\tau}).$

In a last step of this subsection we again consider minconvex lists. We use the definition of $\alpha$ on page~15 with the little difference that here $m\leq {n \choose 2}$. 

\begin{Definition}
Let $\tau: |V| \mapsto |V|$ be an arbitrary permutation and $\alpha_{\tau}$ be a permutation of integer list $\alpha.$ We call a list 
$(\alpha,\alpha_{\tau})$ \emph{minconvex list}. 
\end{Definition}

Now, we are able to prove the general result for minconvex digraphic lists.

\begin{corollary}\label{Korrolar:opposedMinconvexSequences}
Let $(a,b)$ be an arbitrary digraphic list with nonincreasing $a$ and $\sum_{i=1}^{n}a_i=m.$ Then we find for the opposed minconvex list $(\alpha,\alpha_{\sigma})$ that $N_2(\alpha,\alpha_{\sigma})\geq N_2(a,b).$
\end{corollary}

\begin{proof}
When $(a,b)$ is a minconvex list, then we obtain the claim by Theorem~\ref{Theorem:opposedSequences}. So let us assume that $(a,b)$ is not a minconvex list. We transform list $(a,b)$ into the minconvex list $(\alpha,\alpha_{\tau})$ by the following series of lists $(a,b),(\alpha,b),(\alpha_{\tau},b_{\tau}),(\alpha_{\tau},\alpha),(\alpha,\alpha_{\tau})$ where $\tau$ is a permutation such that $b_{\tau}$ is nonincreasing. Since by Theorem~\ref{Theorem:MinconvexSequence} $\alpha \prec a$ and $\alpha \prec b_{\tau}$ and $\alpha$ is nonincreasing by definition we get by Corollary~\ref{Corollary:NumberDigraphRealizationsNonincreasing} that $N_2(a,b)\leq N_2(\alpha,b)$ and $N_2(\alpha_{\tau},b_{\tau})\leq N_2(\alpha_{\tau},\alpha)$ if we switch the components in the last pair of lists. Clearly, $N_2(\alpha_{\tau},b_{\tau})= N_2(\alpha,b)$ and $N_2(\alpha_{\tau},\alpha)= N_2(\alpha,\alpha_{\tau})$. In summery we obtain $N_2(\alpha,\alpha_{\tau})\geq N_2(a,b).$ Since list $(\alpha,\alpha_{\sigma})$ is the opposed list of $(\alpha,\alpha_{\tau})$ we get with Theorem~\ref{Theorem:opposedSequences} the claim.
\end{proof}

\subsection{The number of graph realizations and majorization}

The connection between the number of graph realizations and majorization can easily be proved using the results of the last subsection. For that we consider instead the graphic list $a$ digraphic list $(a,a)$ and restrict the set of all digraph realizations to symmetric digraphs, i.e. for each arc $(v,w)\in A(G)$ also exists arc $(w,v) \in A(G).$ The set of all graph realizations we denote by $R_3(a)$ and we set $N_3(a):=|R_3(a)|.$ The set of all symmetric digraph realizations of $(a,a)$ we denote by $\hat{R}_2(a,a)$ and set $\hat{N}_2(a,a):=|\hat{R}_2(a,a)|.$ Clearly, we cannot use the notion of shifts of the last two subsections, because this would result in a non-symmetric digraph realization. On the other hand, we can simply distinguish between \emph{horizontal (i,j)-shifts} and \emph{vertical $(i,j)$-shifts.} Horizontal $(i,j)$-shifts for an adjacency matrix $A'$ of $G' \in R_3(a')$ are exactly the $(i,j)$-shifts of Definition \ref{Definition:ShiftDigraph}. Vertical $(i,j)$-shifts correspond to horizontal $(i,j)$-shifts in the sense that the symmetric entries have to be shifted in the corresponding columns. (Note that with this definition loops will be excluded.) Clearly, if we always apply horizontal and vertical shifts simultaneously we again yield a symmetric realization and so a graph realization. As in the Proposition~\ref{PropositionSymmetrischeDifferenzdigraphRealization} two symmetric digraph realizations are $(i,j)$-adjacent if there symmetric difference can be found on the $i$th and $j$th columns of their adjacency matrices. Clearly, there exist the corresponding symmetric difference on their rows. We use the notion of $(i,j)$-adjacency in this new sense. $M_{i,j}(a,a)$ denotes here the set of all \emph{symmetric digraph realizations} of $(a,a)$ which possess at least one $(i,j)$-adjacent symmetric digraph realization in $\hat{R}_2(a,a)$. It turns out, that the case of the graph realization problem is more simple than the case of the digraph realization problem. One reason is that each $(a,a)$ is lexicographically nonincreasing if $a$ is nonincreasing. Another point is the absence of difficult cases as they appear in the proof of Theorem~\ref{Theorem:CountDigraphRealizationsLexicographic}. 

\begin{thm}\label{Theorem:CountGRaphRealizations}
Let $a$ be a nonincreasing integer list and $a'$ an integer list such that $a'$ yields $a$ by a unit $(i,j)$-transfer. Then it follows $N_3(a)\geq N_3(a').$ If $a'$ is a graphic list, then we have  $N_3(a) > N_3(a').$
\end{thm}

\begin{proof}
We consider lists $(a,a)$ and $(a',a').$ We distinguish for $G' \in \hat{R}_2(a',a')$ between the two cases $G' \in M_{i,j}(a',a')$ and $G' \in  \hat{R}_2(a',a')\setminus M(i,j)(a',a')$. Consider the following two possible schematic matrices for these cases.

\begin{description}
\item[case 1:] $G' \in \hat{R}_2(a',a')\setminus M_{i,j}(a',a').$ We consider a schematic picture for different $c.$
$A':=$\[
\begin{array}{ccccccccccccccccccc}
          &           && 1 && \ldots &c&& i             && \ldots && j             && \ldots && n                                      \\
1         & \LiKl{12} &&   &&       & && \ob{1}        &&        && \ob{0}        &&        &&   & \ReKl{12} & \rdelim\}{4}{0pt}[$c$]\\
\bigdots  &           &&   &&        &&& \mi{1}        &&        && \mi{0}                                                            \\
          &           &&   &&        &&& \mi{\myvdots} &&        && \mi{\myvdots}                                                     \\
c        &           &&   &&        &&& \un{1}        &&        && \un{0}                                                            \\
i         &           && 1  && \ldots &1  && \ob{0}        &&   && \ob{a'_{ij}}                                                      \\
j         &           && 0  && \ldots &0  && \un{a'_{ji}}  &&   && \un{0}                                                            \\
\bigdots  &           &&   &&        &&& \ob{1}        &&        && \ob{1}                                                            \\
          &           &&   &&        &&& \mi{\myvdots} &&        && \mi{\myvdots}                                                     \\
          &           &&   &&        &&& \un{1}        &&        && \un{1}                                                            \\
          &           &&   &&       & && \ob{0}        &&        && \ob{0}                                                            \\
          &           &&   &&       & && \mi{\myvdots} &&        && \mi{\myvdots}                                                     \\
n         &           &&   &&        &&& \un{0}        &&        && \un{0}                                                            \\
          &           &&   &&        &&& a_i'          &&        && a_j'
\end{array}
\]
Since, $A'$ is symmetric we have the two possibilities a) $a'_{ij}=0$ and $a'_{ji}=0$ and b) $a'_{ij}=1$ and $a'_{ji}=1.$ Hence, we have at least two possible vertical and two possible horizontal $(i,j)$-shifts. It follows that the number of symmetric digraph realizations of $(a,a)$ is at least twice the cardinality of $\hat{R}_2(a',a')\setminus M_{i,j}(a',a').$

\item[case 2:] $G' \in   M(i,j.)$ Again we consider all possible scenarios for pairs $(c,d)$ see the proof of Proposition \ref{Proposition:NumberAdjacentRealizations}. We show a schematic picture.

$A':=$\[
\begin{array}{ccccccccccccccccccc}
          &           && 1 && \ldots &c &\ldots&&& i             &&  j             && \ldots && n                                      \\
1         & \LiKl{15} &&   &&      &&   &&& \ob{1}               && \ob{0}        &&        &&   & \ReKl{15} & \rdelim\}{4}{0pt}[$c$]\\
\bigdots  &           &&   &&   &&     &&& \mi{1}                && \mi{0}                                                            \\
          &           &&   &&      &&  &&& \mi{\myvdots}         && \mi{\myvdots}                                                     \\
c        &           &&   &&      &&   &&& \un{1}               && \un{0}                                                            \\
\bigdots  &           &&   &&     &&   &&& \ob{0}                && \ob{1}        &&        &&   &           & \rdelim\}{3}{0pt}[$d$]\\
          &           &&   &&    &&   & && \mi{\myvdots}         && \mi{\myvdots}                                                     \\
c+d     &           &&   &&      &&   & && \un{0}              && \un{1}                                                            \\
i         &           && 1  &&  \ldots &1  &0&\ldots   &0& \ob{0}              && \ob{a'_{ij}}                                                      \\
j         &           && 0  &&   \ldots& 0  &1&\ldots &1& \un{a'_{ji}}         && \un{0}                                                            \\
\bigdots  &           &&   &&    &&   & && \ob{1}                && \ob{1}                                                            \\
          &           &&   &&    &&   & && \mi{\myvdots}         && \mi{\myvdots}                                                     \\
          &           &&   &&    &&    &&& \un{1}                && \un{1}                                                            \\
          &           &&   &&    &&    & && \ob{0}               && \ob{0}                                                            \\
          &           &&   &&   &&     & && \mi{\myvdots}        && \mi{\myvdots}                                                     \\
n         &           &&   &&   &&    & && \un{0}                && \un{0}                                                            \\
          &           &&   &&   &&    & && a_i'                  && a_j'
\end{array}
\]
Note that $d:=c-\ell$ and $d \geq 1.$ Since, $A'$ is symmetric we have the two possibilities a) $a'_{ij}=0$ and $a'_{ji}=0$ and b) $a'_{ij}=1$ and $a'_{ji}=1.$ Hence, we have at least three possible vertical and three possible $(i,j)$-shifts and so $\ell \geq 2.$ 
In both cases there do exist ${2c-\ell \choose c}$ digraph realizations for $(a',a')$ and ${2c-\ell \choose c-\ell+1}$ digraph realizations of $(a,a).$ With Proposition \ref{Proposition:PascalsTriangle} we have ${2i'-l \choose i'-l+1} > {2i'-l \choose i'}.$
\end{description}
Since, there exists a bijective mapping between  symmetric digraphic lists and graphic lists, case 1 and case 2 lead to $N_3(a)>N_3(a').$
\end{proof}

\begin{corollary}\label{Korollar:CountGraphRealizations}
Let $a$ be a nonincreasing integer list, $a'$ be an integer list, $a \prec a'$ and $a\neq a'.$ Then the number of graph realizations of $a$ is larger than the number of graph realizations of $a'$, i.e., $N_3(a) > N_3(a').$
\end{corollary}

\begin{proof}
There exists at least one transfer path $a':=a^1,\dots,a^r=:s$ with $a^{i+1} \prec a^{i}$ such that $a^i$ yields $a^{i+1}$ by a unit transfer see Theorem~\ref{Theorem:AignerTriesch}. We show by induction on $r$ the correctness of the claim. For $r=2$ we apply Theorem~\ref{Theorem:CountGRaphRealizations} and get $N_3(a')<N_3(a).$ Let us now assume $r>2.$ We consider the transfer path $a^2,\dots, a.$ With our induction hypothesis we can conclude $N_3(a^2)<N_3(a).$ For $a'$ and $a^2$ we apply again Theorem~\ref{Theorem:CountGRaphRealizations}. This yields $N_3(a')\leq N_3(a^2)<N_3(a).$
\end{proof}
  
Similar to the loop-digraph case a transfer path for graphic lists is strictly monotone with respect to the number of realizations. This is not always the case for digraphic lists. It is possible to find an analogous result for special types of transfer paths like in Corollary \ref{Corollar:exponentiellerTransferPath}, i.e.\ one can find exponential lower bounds for the number of graph realizations for some transfer paths. Finally, we again find a result for minconvex graphic lists. 

\begin{corollary}\label{Korrolar:GraphMinconvexSequences}
Let $a$ be a nonincreasing graphic list with $\sum_{i=1}^{n}a_i=m$. Then we find for the minconvex list $\alpha$ that the number of graph realizations of $\alpha$ is larger or equals the number of graph realizations of $a$, $N_3(\alpha)\geq N_3(a).$
\end{corollary}

\begin{proof}
With Theorem~\ref{Theorem:MinconvexSequence} we find that $\alpha \prec a.$ Applying Corollary~\ref{Korollar:CountGraphRealizations} we get $N_3(\alpha)\geq N_3(a).$
\end{proof}

\subsection*{Acknowledgement.} Many thanks to my colleage Ivo Hedtke for beautifully typesetting in LaTeX the matrices in the proofs of 
Propositions~\ref{Proposition:NumberAdjacentloopRealizations}, Proposition~\ref{Proposition:NumberAdjacentRealizations} and 
Theorem~\ref{Theorem:CountDigraphRealizationsLexicographic}.

\providecommand{\bysame}{\leavevmode\hbox to3em{\hrulefill}\thinspace}
\providecommand{\MR}{\relax\ifhmode\unskip\space\fi MR }
\providecommand{\MRhref}[2]{%
  \href{http://www.ams.org/mathscinet-getitem?mr=#1}{#2}
}
\providecommand{\href}[2]{#2}

\end{document}